\documentclass[10 pt]{amsart}
\usepackage{amsmath,amssymb,graphicx,color,amsthm}
\def\zz{\mathbb{Z}} %definitiile la multimile Z,R,N
\def\rr{\mathbb{R}}
\def\nn{\mathbb{N}}
\def\cc{\mathbb{C}}

\newtheorem{theorem}{Theorem}[section]

\newtheorem{proposition}[theorem]{Proposition}

\newtheorem{remark}{\it Remark\/}

\title[Quadratic classical FEM approximation of waves: propagation, observation and control]{On the quadratic finite element approximation of 1-D waves: propagation, observation and control}
\author{Aurora Marica}
\address{Aurora Marica$^{a}$, Enrique Zuazua$^{a,b}$\medskip\hfill\break\indent$^{a}$BCAM - Basque Center for Applied Mathematics,
\hfill\break\indent Bizkaia Technology Park 500, 48160, Derio, Basque Country, Spain \medskip\hfill\break\indent
$^{b}$Ikerbasque - Basque Foundation for Science,
\hfill\break\indent Alameda Urquijo 36-5, Plaza Bizkaia, 48011, Bilbao, Basque Country, Spain\medskip}
\email{marica@bcamath.org, zuazua@bcamath.org}
\urladdr{www.bcamath.org/marica/, www.bcamath.org/zuazua/}
\author{Enrique Zuazua}
\thanks{Both authors were partially supported by the Grant MTM2008-03541 of the MICINN,
Spain, Project PI2010-04 of the Basque Government, the ERC Advanced Grant FP7-246775 NUMERIWAVES and the
ESF Research Networking Programme OPTPDE}
\begin{document}
\maketitle

\begin{abstract} We study the propagation, observation and control properties of the $1-d$ wave equation on a bounded interval discretized in space using the
quadratic $P_2$  - classical finite element approximation. A careful Fourier analysis of the discrete wave dynamics reveals two different branches in the spectrum:
the \textit{acoustic} one, of physical nature,
and the \textit{optic} one, related to the perturbations that this second-order finite element approximation introduces with respect to the $P_1$ one.
On both modes there are high frequencies with vanishing group velocity as the mesh size tends to zero. This shows that the classical property of continuous waves of being observable from the boundary fails to be uniform for this discretization scheme. As a consequence of this, the controls of the discrete waves may blow-up as the mesh size tends to zero. To remedy these high frequency pathologies, we design filtering mechanisms based on a bi-grid algorithm for which one can recover the uniformity of the observability constant
in a finite time and, consequently, the possibility to control with uniformly bounded $L^2$ - controls appropriate projections of the solutions. This also allows showing that, by relaxing the control requirement, the controls are uniformly bounded and converge to the continuous ones as the mesh size tends to zero.
\end{abstract}
\section{Introduction, problem formulation and main results}\label{SectIntro}
Consider the $1-d$ wave equation with non-homogeneous boundary conditions:
\begin{equation}\left\{\begin{array}{l}y_{tt}(x,t)-y_{xx}(x,t)=0,\ x\in(0,1),\ t>0,\\y(0,t)=0,\ y(1,t)=v(t),\ t>0,\\
y(x,0)=y^0(x),\ y_t(x,0)=y^1(x),\ x\in(0,1).\end{array}\right.\label{contwavecontrolled}\end{equation}
System (\ref{contwavecontrolled}) is said to be \textit{exactly controllable} in time $T\geq 2$ if, for all $(y^0,y^1)\in L^2\times H^{-1}(0,1)$, there exists a control function
$v\in L^2(0,T)$ such that the solution of (\ref{contwavecontrolled}) can be driven at rest at time $T$, i.e.
\begin{equation}y(x,T)=y_t(x,T)=0.\label{ContNullControl}\end{equation}

We also introduce the adjoint $1-d$ wave equation with homogeneous boundary conditions:
\begin{equation}\left\{\begin{array}{l}u_{tt}(x,t)-u_{xx}(x,t)=0,\ x\in(0,1),\ t>0,\\u(0,t)=u(1,t)=0,\ t>0,\\
u(x,T)=u^0(x),\ u_t(x,T)=u^1(x),\ x\in(0,1).\end{array}\right.\label{contwaveadjoint}\end{equation}

This system is well known to be well posed in the energy space $\mathcal{V}:=H_0^1\times L^2(0,1)$ and the energy below is conserved in time:
$$\mathcal{E}(u^0,u^1):=\frac{1}{2}(||u(\cdot,t)||_{H_0^1(0,1)}^2+||u_t(\cdot,t)||_{L^2(0,1)}^2).$$

The Hilbert Uniqueness Method (HUM) introduced in \cite{Lio} allows showing that the property of \textit{exact controllability} for (\ref{contwavecontrolled}) is equivalent to the \textit{boundary observability} of (\ref{contwaveadjoint}). The \textit{observability property} of the wave equation ensures that the following \textit{observability inequality} holds for all solutions of (\ref{contwaveadjoint}) provided $T\geq 2$:
\begin{equation}\mathcal{E}(u^0,u^1)\leq C(T)\int\limits_0^T |u_x(1,t)|^2\,dt.\label{contobsineq}\end{equation}

The best constant $C(T)$ in (\ref{contobsineq}) is the so-called \textit{observability constant}. The observability time $T$ has to be larger than the characteristic time $T^{\star}:=2$ which is needed by any solution associated to initial data $(u^0,u^1)$ supported in a very narrow neighborhood of $x=1$ to travel along the characteristic rays $x(t)=x-t$, get to the boundary
$x=0$ and come back to the boundary $x=1$ along the characteristics $x(t)=x+t$.

It is also well known that for all $T>0$ and all solutions $u$ of the adjoint problem (\ref{contwaveadjoint}) with initial data $(u^0,u^1)\in\mathcal{V}$, the following
\textit{admissibility inequality} holds:
\begin{equation}c(T)\int\limits_0^T|u_x(1,t)|^2\,dt\leq\mathcal{E}(u^0,u^1),\label{contDirectInequality}\end{equation}
so that for all $T\geq 2$, $||\partial_xu(1,\cdot)||_{L^2(0,T)}$ and $||(u^0,u^1)||_{\mathcal{V}}$ are equivalent norms.

As a consequence of these results, it is easy to see that for all $(y^0,y^1)\in L^2\times H^{-1}(0,1)$, there exists a control $v\in L^2(0,T)$ driving the solution of (\ref{contwavecontrolled}) to the rest at $t=T$, i.e. such that (\ref{ContNullControl}) holds.  This turns out to be equivalent to the fact that \begin{equation}\int\limits_0^Tv(t)u_x(1,t)\,dt=\langle (y^1,-y^0),(u(\cdot,0),u_t(\cdot,0))\rangle_{\mathcal{V}',\mathcal{V}},
\label{identitycontcontrol}\end{equation}
for all solutions $u$ of (\ref{contwaveadjoint}), where $\langle\cdot,\cdot\rangle_{\mathcal{V}',\mathcal{V}}$ is the
duality product between $\mathcal{V}'=H^{-1}\times L^2(0,1)$ and $\mathcal{V}$.

The HUM control $v$, the one of minimal $L^2(0,T)$-norm, has the explicit
form
\begin{equation}v(t)=\tilde{v}(t):=\tilde{u}_x(1,t),\label{contHUMcontrol}\end{equation}
where $\tilde{u}(x,t)$ is the solution corresponding to the minimum $(\tilde{u}^0,\tilde{u}^1)\in\mathcal{V}$ of the quadratic functional
\begin{equation}\mathcal{J}(u^0,u^1):=\frac{1}{2}\int\limits_0^T|u_x(1,t)|^2\,dt-
\langle (y^1,-y^0),(u(\cdot,0),u_t(\cdot,0))\rangle_{\mathcal{V}',\mathcal{V}}.\label{contfunctional}\end{equation}

The effect of substituting the continuous controlled wave equation (\ref{contwavecontrolled}) or
the corresponding adjoint problem (\ref{contwaveadjoint}) by discrete models has been intensively studied during the last years, starting from some
simple numerical schemes on uniform meshes like finite differences or linear $P_1$ - finite element methods in \cite{InfZua} and, more recently, more
complex schemes like the discontinuous Galerkin ones in \cite{MarZuaDG}. In all these cases, the convergence of the approximation scheme
in the classical sense of the numerical analysis does not suffice to guarantee that
the sequence of discrete controls converges to the continuous ones, as one could expect. This is due to the fact that there are classes of initial data for the
discrete adjoint problem generating high frequency wave packets propagating at a very low group velocity and that, consequently, are impossible to be observed from the boundary of the domain during a finite time, uniformly as the mesh-size parameter tends to zero. This leads to the divergence of the discrete observability constant as
the mesh size tends to zero. High frequency pathological phenomena have also been observed for numerical approximation schemes of other models, like the linear Schr\"{o}dinger
equation (cf. \cite{IgZuaSch}), in which one is interested in the uniformity of the so-called \textit{dispersive estimates}, which play an important role in the study of the well-posedness of the non-linear models.

Several \textit{filtering techniques} have been designed to face these high frequency pathologies: the \textit{Fourier truncation method} (cf. \cite{InfZua}), which simply eliminates all the high frequency Fourier components propagating non-uniformly;
the \textit{bi-grid algorithm} introduced in \cite{GloWell} and \cite{GloLiLio} and rigorously studied in \cite{IgZuaWave},
\cite{LorMeh} or \cite{NegZuaCR} in the context of the finite differences semi-discretization of the $1-d$ and $2-d$ wave equation and of the Schr\"{o}dinger equation
(cf.\cite{IgZuaSch}), which consists in taking initial data with
slow oscillations obtained by linear interpolation from data given on a coarser grid; and the \textit{numerical viscosity} method, which by adding a suitable dissipative mechanism damps out the spurious high frequencies (\cite{MicuVisc}, \cite{TebZuaBoundaryDamp}). We should emphasize that the \textit{mixed finite element method} analyzed in \cite{CasMicu} is, as far as we know, the unique method that preserves the propagation and controllability properties of
the continuous model uniformly in the mesh size without requiring any extra filtering. The interested reader is referred to the survey articles \cite{ErvZuaSurv} and \cite{ZuaPOC} for a  presentation of the development and the state of the art in this topic.

The purpose of the present paper is to analyze the behavior of the quadratic $P_2$ - finite element
semi-discretization of the problems (\ref{contwavecontrolled}) and (\ref{contwaveadjoint}) from the
 uniform controllability/observability point of view. In \textbf{Section \ref{SectIntroP2}} we introduce in a rigorous way the discrete analogue
of (\ref{contwavecontrolled}) and (\ref{contwaveadjoint}) and explain the minimization process generating the discrete controls. In \textbf{Section
\ref{SectFour}} we analyze the spectral problem associated to this discrete model and reveal the co-existence of two main types of Fourier modes:
an \textit{acoustic} one, of physical nature, related to the nodal components of the numerical solution, and an \textit{optic} one, of spurious nature,
related to the curvature with which the quadratic finite element method perturbs the linear approximation. We also study finer properties of the spectrum, for
example the \textit{spectral gap}, identifying three regions of null gap: the highest frequencies on both acoustic and optic modes and the lowest
frequencies on the optic one. The content of this section is related to previously existing work. For instance, the dispersive properties of higher-order finite element methods have been analyzed in \cite{Ains}
in the setting of the Helmholtz equation. An explicit form of the acoustic dispersion relation was obtained for approximations of arbitrary order.
It was shown that the numerical dispersion displays three different types of behavior,
depending on the approximation order relative to the mesh-size and the wave number. In \textbf{Section \ref{SectBoundObsEig}} we obtain some spectral identities allowing us to analyze the discrete observability inequality
for the adjoint system. In \textbf{Section \ref{SectIngham}} we show that the \textit{Fourier truncation} of the three pathological regions of the spectrum leads to
an uniform observability inequality. In \textbf{Section \ref{SectBigrid}} we prove that a filtering mechanism consisting in, firstly, considering piecewise linear initial data and, secondly,
preconditionning the nodal components by a bi-grid algorithm guarantees uniform observability properties. Within the proof, we use a classical \textit{dyadic decomposition
argument} (cf. \cite{IgZuaWave}), which mainly relies on the fact that for this class of initial data the total energy can be bounded by the energy of the projection
on the low frequency components of the acoustic dispersion relation. We should emphasize that our results are finer than the ones in \cite{ErvSpectralWave} or \cite{ErvSpectralSch}, where one obtains uniform observability properties for finite element approximations of any order, but by filtering the Fourier modes much under the critical scale $1/h$. Here we only consider the particular case of quadratic finite element approximation on $1-d$ meshes, but we get to the critical filtering scale $1/h$. Note however that the results in \cite{ErvSpectralWave} and \cite{ErvSpectralSch} apply in the context of non-uniform grids as well. In \textbf{Section \ref{SectConvergence}} we present the main steps of the proof of the convergence of the discrete control problem under the assumption that the initial data in the corresponding adjoint problem are filtered through a Fourier truncation or a bi-grid algorithm. \textbf{Section \ref{SectOpenPbms}} is devoted to present the conclusions of the paper and some related open problems.

\section{The $P_2$ - finite element approximation of $1-d$ waves}\label{SectIntroP2} Let $N\in\nn$, $h=1/(N+1)$ and $0=x_0<x_j<x_{N+1}=1$ be the \textit{nodes} of an \textit{uniform grid} of the interval
$[0,1]$, with $x_j=jh$, $0\leq j\leq N+1$, constituted by the subintervals $I_j=(x_j,x_{j+1})$, with $0\leq j\leq N$. We also define the
\textit{midpoints} $x_{j+1/2}=(j+1/2)h$ of this grid, with $0\leq j\leq N$. Let us introduce the space $\mathcal{P}_p(a,b)$ of polynomials of order $p$
on the interval $(a,b)$ and the \textit{space of piecewise quadratic and continuous functions}
$\mathcal{U}_h:=\{u\in H_0^1(0,1)\mbox{ s.t. }u|_{I_j}\in\mathcal{P}_2(I_j),\ 0\leq j\leq N\}$.
The space $\mathcal{U}_h$ can be written as $\mathcal{U}_h=\mbox{span}\{\phi_j,1\leq j\leq N\}\oplus\mbox{span}\{\phi_{j+1/2},0\leq j\leq N\}$, where the
two classes of basis functions are explicitly given below (see Fig. \ref{figbasis})
\begin{equation}\begin{array}{l}\phi_j(x)=\left\{\begin{array}{l}\frac{2}{h^2}(x-x_{j-1/2})(x-x_{j-1}),\ x\in I_{j-1},\\
\frac{2}{h^2}(x-x_{j+1/2})(x-x_{j+1}),\ x\in I_j,\\0,\ \mbox{otherwise}\end{array}\right.\mbox{ and  }
\phi_{j+1/2}(x)=\left\{\begin{array}{l}-\frac{4}{h^2}(x-x_j)(x-x_{j+1}),\ x\in I_j,\\0,\ \mbox{otherwise.}\end{array}\right.\end{array}\nonumber\end{equation}

\begin{figure}
 \begin{center}\includegraphics[width=5.5cm,height=4.5cm]{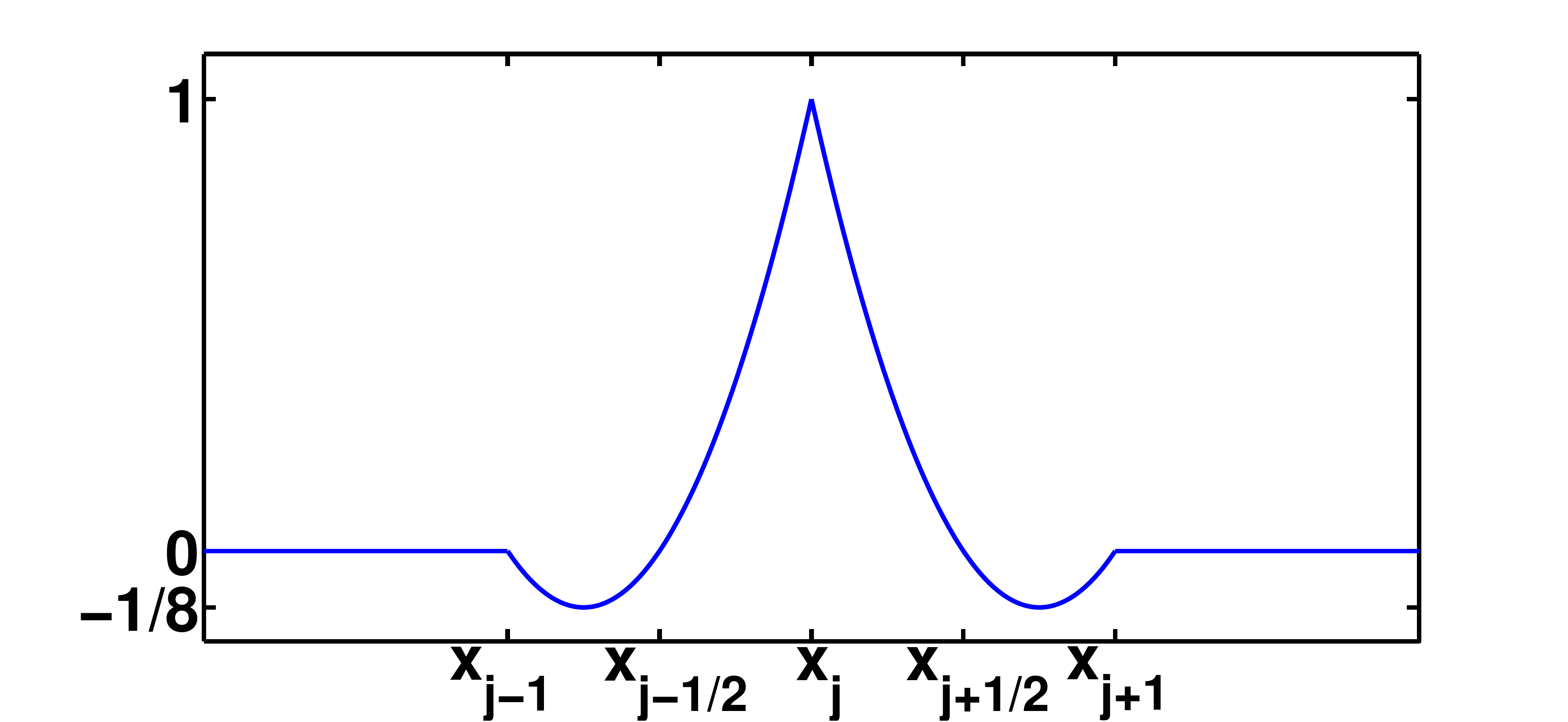}\includegraphics[width=5.5cm,height=4.5cm]{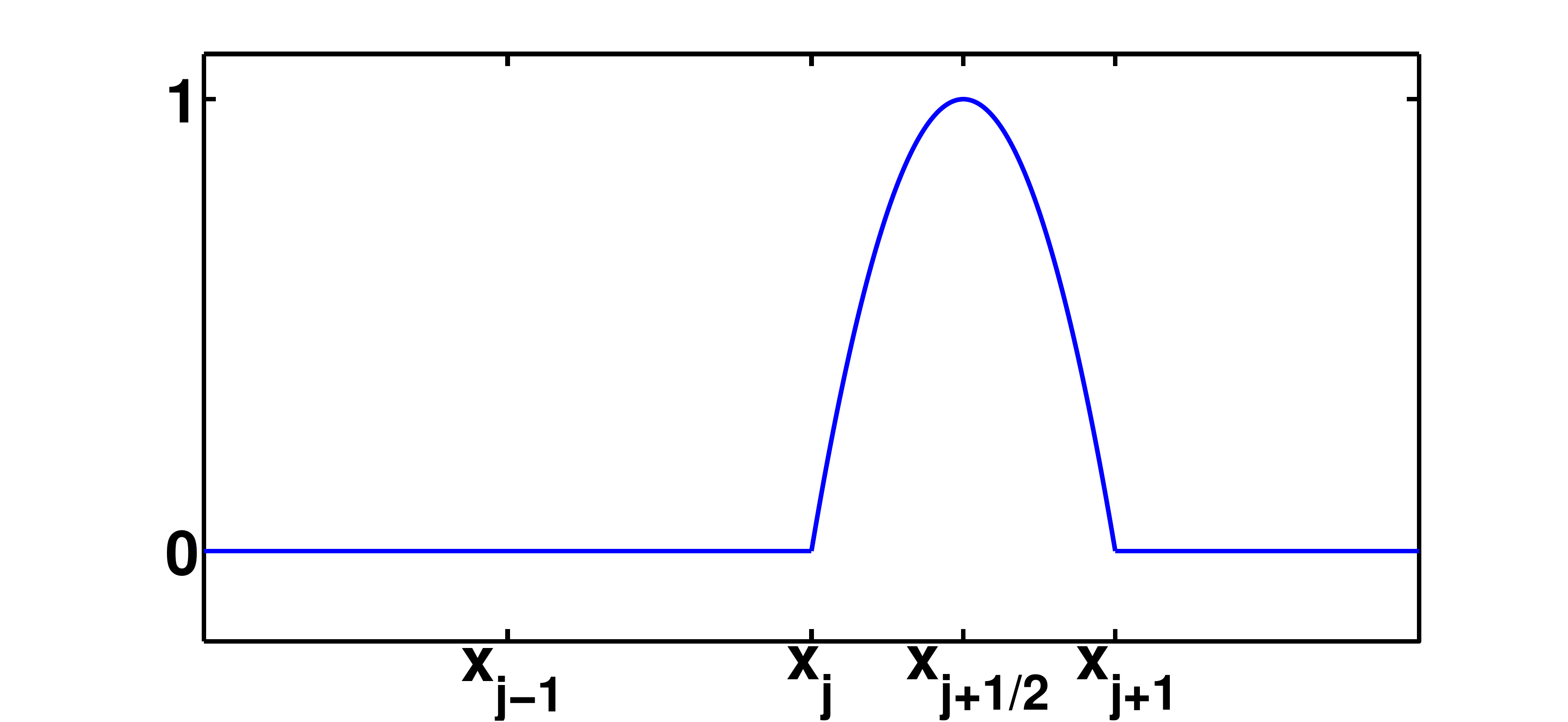}\end{center}
  \caption{The basis functions: $\phi_j$ (left) and $\phi_{j+1/2}$ (right).}\label{figbasis}
\end{figure}
\noindent The quadratic approximation of the adjoint problem (\ref{contwaveadjoint}) is as follows:
\begin{equation}\left\{\begin{array}{l}\mbox{Find }u_h(\cdot,t)\in \mathcal{U}_h\mbox{ s.t. }\frac{d^2}{dt^2}( u_h(\cdot,t),
\varphi)_{L^2(0,1)}
+(u_h(\cdot,t),\varphi)_{H_0^1(0,1)}=0,\forall\varphi\in\mathcal{U}_h,\\u_h(x,T)=u_h^0(x),u_{h,t}(x,T)=u_h^1(x),\ x\in(0,1).\end{array}\right.
\label{p2waveadjointvar}\end{equation}
Since $u_h(\cdot,t)\in\mathcal{U}_h$, it admits the decomposition $u_h(x,t)=\sum_{j=1}^{2N+1}U_{j/2}(t)\phi_{j/2}(x)$. The function
$u_h(\cdot,t)$ can be identified with the vector of its coefficients, $\mathbf{U}_h(t):=(U_{j/2}(t))_{1\leq j\leq 2N+1}$ (in the sequel, all vectors under consideration will be column vectors). Thus,
using $\varphi=\phi_{j/2}$, $1\leq j\leq 2N+1$, as test functions in (\ref{p2waveadjointvar}), system (\ref{p2waveadjointvar}) can be written as the following system of $2N+1$ second-order linear differential equations (ODEs):
\begin{equation}M_h\mathbf{U}_{h,tt}(t)+S_h\mathbf{U}_h(t)=0,\ \mathbf{U}_h(T)=\mathbf{U}^0_h,\
\mathbf{U}_{h,t}(T)=\mathbf{U}_h^1,\label{p2adjoint}\end{equation}
where $M_h$ and $S_h$ are the following $(2N+1)\times (2N+1)$ \textit{pentha-diagonal mass} and \textit{stiffness} matrices
$$M_h=\left(\begin{array}{ccccccccccc}
                        \frac{8h}{15} & \frac{h}{15} & 0 & 0 & 0 & 0 & \cdots & 0 & 0 & 0 & 0 \\
                        \frac{h}{15} & \frac{4h}{15} & \frac{h}{15} & -\frac{h}{30} & 0 & 0 & \cdots & 0 & 0 & 0 & 0 \\
                        0 & \frac{h}{15} & \frac{8h}{15} & \frac{h}{15} & 0 & 0 & \cdots & 0 & 0 & 0 & 0 \\
                        0 & -\frac{h}{30} & \frac{h}{15} & \frac{4h}{15} & \frac{h}{15} & -\frac{h}{30} & \cdots & 0 & 0 & 0 & 0 \\
                        \cdots & \cdots & \ddots & \ddots & \ddots & \ddots & \ddots & \cdots & \cdots & \cdots & \cdots \\
                        0 & 0 & 0 & 0 & 0 & 0 & \cdots & -\frac{h}{30} & \frac{h}{15} & \frac{4h}{15} & \frac{h}{15} \\
                        0 & 0 & 0 & 0 & 0 & 0 & \cdots & 0 & 0 & \frac{h}{15} & \frac{8h}{15} \\
                      \end{array}
                    \right)$$
and
$$S_h=\left(
                      \begin{array}{ccccccccccc}
                        \frac{16}{3h} & -\frac{8}{3h} & 0 & 0 & 0 & 0 & \cdots & 0 & 0 & 0 & 0 \\
                        -\frac{8}{3h} & \frac{14}{3h} & -\frac{8}{3h} & \frac{1}{3h} & 0 & 0 & \cdots & 0 & 0 & 0 & 0 \\
                        0 & -\frac{8}{3h} & \frac{16}{3h} & -\frac{8}{3h} & 0 & 0 & \cdots & 0 & 0 & 0 & 0 \\
                        0 & \frac{1}{3h} & -\frac{8}{3h} & \frac{14}{3h} & -\frac{8}{3h} & \frac{1}{3h} & \cdots & 0 & 0 & 0 & 0 \\
                        \cdots & \cdots & \ddots & \ddots & \ddots & \ddots & \ddots & \cdots & \cdots & \cdots & \cdots \\
                        0 & 0 & 0 & 0 & 0 & 0 & \cdots & \frac{1}{3h} & -\frac{8}{3h} & \frac{14}{3h} & -\frac{8}{3h} \\
                        0 & 0 & 0 & 0 & 0 & 0 & \cdots & 0 & 0 & -\frac{8}{3h} & \frac{16}{3h} \\
                      \end{array}
                    \right).$$
We introduce the discrete analogue of $H_0^1(0,1)$, $L^2(0,1)$ and $H^{-1}(0,1)$ spaces to be
$$\mathcal{H}_h^i:=\{\mathbf{F}_h=(F_{j/2})_{1\leq j\leq 2N+1}\in\cc^{2N+1}\mbox{ s.t. }||\mathbf{F}_h||_{h,i}<\infty\},\quad
i=-1, 0,1.$$
For the elements of the space $\mathcal{H}_h^1$ we have to impose the additional requirement $F_0=F_{N+1}=0$. The inner products defining the discrete spaces $\mathcal{H}_h^i$, $i=-1,0,1$, are given by
$$(\mathbf{F}_h,\mathbf{G}_h)_{h,i}:=((M_hS_h^{-1})^{1-i}S_h\mathbf{F}_h,\mathbf{G}_h)_{\cc^{2N+1}},$$
and the norms are given by $||\mathbf{F}_h||_{h,i}^2:=(\mathbf{F}_h,\mathbf{F}_h)_{h,i}$, for all $i=-1,0,1$.
Here, $(\cdot,\cdot)_{\cc^{2N+1}}$ is the inner product in the Euclidian space $\cc^n$, defined by $(\mathbf{F}_h,\mathbf{G}_h)_{\cc^{2N+1}}:=
\sum_{k=1}^{2N+1}F_{k/2}\overline{G}_{k/2}$ (the overline symbol denotes complex conjugation).

For $f_h\in\mathcal{U}_h$ with coefficients $\mathbf{F}_h=(F_{j/2})_{1\leq j\leq 2N+1}\in \mathcal{H}_h^1$, we introduce the following notations
for the three possible discrete derivatives on each nodal point (the \textit{forward}, the \textit{backward} and the \textit{midpoint} one):
\begin{equation}\begin{array}{c}\partial_h^+F_j:=f_{h,x}(x_j+)=-\frac{F_{j+1}-4F_{j+1/2}+3F_j}{h},
\ \partial_h^-F_{j+1}:=f_{h,x}(x_{j+1}-)=\frac{F_j-4F_{j+1/2}+3F_{j+1}}{h},\medskip\\ \partial_hF_{j+1/2}:=f_{h,x}(x_{j+1/2})
=\frac{F_{j+1}-F_j}{h}\end{array}\label{DiscreteOperators}\end{equation}
and the values of $f_h$ at $x_{j+1/4}:=(j+1/4)h$ and $x_{j+3/4}:=(j+3/4)h$, $0\leq j\leq N$:
$$F_{j+1/4}:=f_h(x_{j+1/4})=\frac{3}{8}F_j+\frac{3}{4}F_{j+1/2}-\frac{1}{8}F_{j+1},\ F_{j+3/4}
:=f_h(x_{j+3/4})=-\frac{1}{8}F_j+\frac{3}{4}F_{j+1/2}+\frac{3}{8}F_{j+1}.$$

With these notations, it is easy to check that the $||\cdot||_{h,1}$ and $||\cdot||_{h,0}$ - norms admit the following representations:
\begin{equation}\begin{array}{c}||\mathbf{F}_h||_{h,1}^2=\frac{h}{6}\sum\limits_{j=0}^N\big(|\partial_h^+F_j|^2+4|\partial_hF_{j+1/2}|^2+|\partial_h^-F_{j+1}|^2\big)
\\||\mathbf{F}_h||_{h,0}^2=\frac{h}{90}\sum\limits_{j=0}^N\big(7|F_j|^2+32|F_{j+1/4}|^2+12|F_{j+1/2}|^2+32|F_{j+3/4}|^2+7|F_{j+1}|^2\big).\end{array}\label{discretenorms}
\end{equation}

Set $\mathcal{V}_h:=\mathcal{H}_h^1\times\mathcal{H}_h^0$ and its dual $\mathcal{V}_h':=\mathcal{H}_h^{-1}\times\mathcal{H}_h^0$,  the
duality product $\langle\cdot,\cdot\rangle_{\mathcal{V}_h',\mathcal{V}_h}$ between $\mathcal{V}_h'$ and $\mathcal{V}_h$ being defined as
$\langle(\mathbf{F}_{h,1},\mathbf{G}_{h,1}),(\mathbf{F}_{h,2},\mathbf{G}_{h,2})\rangle_{\mathcal{V}_h',\mathcal{V}_h}:=
(\mathbf{F}_{h,1},\mathbf{F}_{h,2})_{h,0}+(\mathbf{G}_{h,1},\mathbf{G}_{h,2})_{h,0}$.

Problem (\ref{p2adjoint}) is well posed in $\mathcal{H}_h^1\times\mathcal{H}_h^0$. The total energy of its solutions is conserved in time:
\begin{equation}\mathcal{E}_h(\mathbf{U}_h^0,\mathbf{U}_h^1):=\frac{1}{2}(||\mathbf{U}_h(t)||_{h,1}^2+||\mathbf{U}_{h,t}(t)||_{h,0}^2).\label{p2energy}\end{equation}

One of the goals of this paper is to study discrete versions of the observability inequality (\ref{contobsineq}) of the form
\begin{equation}\mathcal{E}_h(\mathbf{U}_h^0,\mathbf{U}_h^1)\leq C_h(T)\int\limits_0^T||B_h\mathbf{U}_h(t)||^2_{\cc^{2N+1}}\,dt,
\label{p2ObservabilityInequalityGeneral}\end{equation}
where $B_h$ is a $(2N+1)\times(2N+1)$ \textit{observability matrix operator}. The observability inequality (\ref{p2ObservabilityInequalityGeneral}) makes sense for rather general matrices $B_h$, corresponding, for instance, to the 
observability from any open subset contained in the spatial domain $(0,1)$. But,  within this paper, we focus on the particular case of \textit{boundary observation operators} $B_h$,
in the sense that they approximate the normal derivative $u_x(x,t)$ of the solution of the continuous adjoint problem (\ref{contwaveadjoint}) at
$x=1$. The main example of such boundary matrix operators $B_h$ that will be
used throughout this paper is as follows:
\begin{equation}
B_{ij}:=\left\{\begin{array}{ll}-\frac{1}{h},&(i,j)=(2N+1,2N),\\0,&\mbox{otherwise.}\end{array}\right.\label{ObservationOperators}\end{equation}

The operator $B_h$ is also the one used for the finite difference semi-discretization in \cite{InfZua}. Let us remark that at the discrete level there are different  ways to approximate the normal derivative of the continuous solution of (\ref{contwaveadjoint}). Since $B_h\mathbf{U}_h(t)$ is a vector and $u_x(x,t)$ is a scalar, the way in which $B_h\mathbf{U}_h(t)$
approximates $u_x(1,t)$ needs to be further explained. Remark that $B_h$ in (\ref{ObservationOperators}) is almost a null matrix, excepting the penultimate component on the last row. The last component of $B_h\mathbf{U}_h(t)$, the only non-trivial one, equals to $u_{h,x}(x_{N+1/2},t)$. The consistency analysis shows that if $f$ is a sufficiently regular function and $f_h$ is its quadratic interpolation, then $f_{h,x}(x_{N+1/2})$ is a first-order approximation of $f_x(1)$.

We are interested  in observability inequalities (\ref{p2ObservabilityInequalityGeneral}) in a finite,
but sufficiently large observability time, say
$T>T^{\star}>0$. In this paper we show that, when working on the whole discrete space $\mathcal{V}_h\times\mathcal{V}_h$, the observability constant $C_h(T)$ blows-up as $h\to 0$, whatever $T>0$ is. One of the main contributions of this paper is to design appropriate subspaces
$\mathcal{S}_h\subset\mathcal{V}_h$ on which the observability constant $C_h(T)$ is uniformly bounded as $h\to 0$.

We will also prove that the discrete version of (\ref{contDirectInequality}) below holds uniformly as $h\to 0$ in the approximate space $\mathcal{V}_h$ for all $T>0$:
\begin{equation}c_h(T)\int\limits_0^T||B_h\mathbf{U}_h(t)||^2_{\cc^{2N+1}}\,dt\leq \mathcal{E}_h(\mathbf{U}_h^0,\mathbf{U}_h^1).
\label{p2DirectInequalityGeneral}\end{equation}

 Once the observability problem is well understood, we are in conditions to address the discrete control problem. For a discrete control function $\mathcal{V}_h$, we consider the following non-homogeneous problem:
\begin{equation}M_h\mathbf{Y}_{h,tt}(t)+S_h\mathbf{Y}_h(t)=-B_h^*\mathbf{V}_h(t),
\quad \mathbf{Y}_h(0)=\mathbf{Y}_h^0,\quad \mathbf{Y}_{h,t}(0)=\mathbf{Y}_h^1. \label{p2controlledpbm}\end{equation}

Here, the superscript $*$ denotes matrix transposition. Multiplying system (\ref{p2controlledpbm}) by any solution $\mathbf{U}_h(t)$ of the adjoint problem (\ref{p2adjoint}), integrating in time and
imposing that at $t=T$ the solution is at rest, i.e.
\begin{equation}\langle(\mathbf{Y}_{h,t}(T),-\mathbf{Y}_h(T)),
(\mathbf{U}_h^0,\mathbf{U}_h^1)\rangle_{\mathcal{V}_h',\mathcal{V}_h}=0,\
\forall (\mathbf{U}_h^0,\mathbf{U}_h^1)\in\mathcal{V}_h,\label{p2ExactControl}\end{equation}
we obtain that $\mathbf{V}_h(t)$ necessarily satisfies the following identity which fully characterizes all possible exact controls $\mathbf{V}_h(t)$:
\begin{equation}\int\limits_0^T(\mathbf{V}_h(t),B_h\mathbf{U}_h(t))_{\cc^{2N+1}}\,dt=
\langle(\mathbf{Y}_h^1,-\mathbf{Y}_h^0),(\mathbf{U}_h(0),\mathbf{U}_{h,t}(0))\rangle_{\mathcal{V}_h',\mathcal{V}_h},\quad
\forall (\mathbf{U}_h^0,\mathbf{U}_h^1)\in\mathcal{V}_h.
\label{identityp2control}\end{equation}

In view of this, we introduce the following discrete version of the quadratic functional (\ref{contfunctional}):
\begin{equation}\mathcal{J}_h(\mathbf{U}_h^0,\mathbf{U}_h^1)=
\frac{1}{2}\int\limits_0^T||B_h\mathbf{U}_h(t)||^2_{\cc^{2N+1}}\,dt-
\langle(\mathbf{Y}_h^1,-\mathbf{Y}_h^0),
(\mathbf{U}_h(0),\mathbf{U}_{h,t}(0))\rangle_{\mathcal{V}_h',\mathcal{V}_h},\label{p2FunctionalGeneral}\end{equation}
$\mathbf{U}_h(t)$ being the solution of the adjoint problem (\ref{p2adjoint}) with initial data
$(\mathbf{U}_h^0,\mathbf{U}_h^1)$
and $(\mathbf{Y}_h^1,\mathbf{Y}_h^0)\in\mathcal{V}_h'$ the initial data to be controlled in (\ref{p2controlledpbm}).

The functional $\mathcal{J}_h$ is continuous and strictly convex. Thus, provided it is coercive (which is actually what the uniform observability inequality   guarantees), it has an unique minimizer $(\mathbf{\tilde{U}}_h^0,\mathbf{\tilde{U}}_h^1)\in\mathcal{S}_h$ whose Euler-Lagrange equations are as follows:
\begin{equation}\int\limits_0^T(B_h\mathbf{\tilde{U}}_h(t),B_h\mathbf{U}_h(t))_{\cc^{2N+1}}\,dt=
\langle(\mathbf{Y}_h^1,-\mathbf{Y}_h^0),
(\mathbf{U}_h(0),\mathbf{U}_{h,t}(0))\rangle_{\mathcal{V}_h',\mathcal{V}_h},\label{p2EulerLagrange}\end{equation}
for all $(\mathbf{U}_h^0,\mathbf{U}_h^1)\in\mathcal{S}_h$ and $\mathbf{U}_h(t)$ the corresponding solution of (\ref{p2adjoint}).
The discrete HUM control is then
\begin{equation}\mathbf{V}_h(t)=\mathbf{\tilde{V}}_h(t):=B_h\mathbf{\tilde{U}}_h(t).\label{p2OptimalControl}\end{equation}

Let us briefly comment the analogies between the identities (\ref{contHUMcontrol}) and (\ref{p2OptimalControl}). As we said, when $B_h$ is
a boundary observability matrix operator, like for example the one in (\ref{ObservationOperators}), $B_h\mathbf{\tilde{U}}_h(t)$ is a vector whose last component $\tilde{v}_h(t)$, the only non-trivial one, approximates the normal derivative of the
solution to the adjoint continuous wave equation (\ref{contwaveadjoint}). Accordingly, the controls $-B_h^*\mathbf{V}_h(t)$ only act on $y_{N}(t)$ when
$\mathbf{V}_h(t)$ is the numerical control obtained by (\ref{p2OptimalControl}). Consequently, the boundary observability operator $B_h$ does not act really at $x=1$, but at $x=x_{N+1/2}$, being
in fact an internal control acting on a single point which is closer and closer to the boundary as the mesh size parameter becomes smaller,
so that in the limit as $h\to 0$ it becomes a boundary control.

Observe also that the control problem we deal with is a coupled system of non-homogeneous ODEs modeling the interaction between the nodal
and the midpoint components. Thus, the node $x_{N+1/2}$ lies in fact
on the boundary of the midpoint component. Consequently, the controls $-B_h^*\mathbf{V}_h(t)$ in (\ref{p2controlledpbm}), with $\mathbf{V}_h(t)$ as in (\ref{p2OptimalControl}), are really natural approximations of the continuous boundary controls $v$ in (\ref{contwavecontrolled}).

\section{Fourier analysis of the $P_2$-finite element method}\label{SectFour}
For the sake of completeness, we recall the spectral analysis of this quadratic finite element method, following \cite{HRS}. The spectral problem associated to the adjoint system (\ref{p2waveadjointvar}) is as follows:
\begin{equation}\mbox{Find }(\Lambda_h,\tilde{\varphi}_h)\in\rr\times\mathcal{U}_h\mbox{ such that }
(\tilde{\varphi}_h,\phi_h)_{H_0^1}=\Lambda_h(\tilde{\varphi}_h,\phi_h)_{L^2},\ \forall\phi_h\in\mathcal{U}_h.\label{p2spectralvar}\end{equation}

Due to the symmetry and the coercivity of the bi-linear forms generated by the scalar products $(\cdot,\cdot)_{L^2}$ and $(\cdot,\cdot)_{H_0^1}$, we have
$\Lambda_h>0$. Let $\mathbf{\tilde{\varphi}}_h=(\tilde{\varphi}_{j/2})_{1\leq j\leq 2N+1}'$ be the components of
the eigenfunction $\tilde{\varphi}_h$. The pair $(\Lambda_h,\mathbf{\tilde{\varphi}}_h)$ is a generalized eigensolution corresponding
to the pair of matrices $(S_h,M_h)$, i.e.
\begin{equation}S_h\mathbf{\tilde{\varphi}}_h=\Lambda_hM_h\mathbf{\tilde{\varphi}}_h.\label{p2spectral}\end{equation}

Consider the normalized eigenvalues $\Lambda:=h^2\Lambda_h$. The system (\ref{p2spectral}) is a pair of two equations:
\begin{equation}\label{p2spectral1}
-\frac{8}{3}\tilde{\varphi}_j+\frac{16}{3}\tilde{\varphi}_{j+1/2}-\frac{8}{3}\tilde{\varphi}_{j+1}-
\Lambda\big(\frac{1}{15}\tilde{\varphi}_j+\frac{8}{15}\tilde{\varphi}_{j+1/2}+\frac{1}{15}\tilde{\varphi}_{j+1}\big)=0,
\end{equation}
for $0\leq j\leq N$, and
\begin{equation}\frac{1}{3}\tilde{\varphi}_{j-1}-\frac{8}{3}\tilde{\varphi}_{j-1/2}+\frac{14}{3}\tilde{\varphi}_j-\frac{8}{3}\tilde{\varphi}_{j+1/2}
+\frac{1}{3}\tilde{\varphi}_{j+1}-
\Lambda\big(-\frac{1}{30}\tilde{\varphi}_{j-1}+\frac{1}{15}\tilde{\varphi}_{j-1/2}+\frac{4}{15}\tilde{\varphi}_j+\frac{1}{15}\tilde{\varphi}_{j+1/2}
-\frac{1}{30}\tilde{\varphi}_{j+1}\big)=0,\label{p2spectral2}\end{equation}
depending on $1\leq j\leq N$ and with the boundary condition $\tilde{\varphi}_0=\tilde{\varphi}_{N+1}=0$.

\textbf{The acoustic and optic modes. }From (\ref{p2spectral1}),  we obtain that for $\Lambda\not=10$, the values of $\mathbf{\tilde{\varphi}}_h$ at the midpoints
can be obtained according to the two neighboring nodal values as follows:
\begin{equation}\tilde{\varphi}_{j+1/2}=\frac{40+\Lambda}{8(10-\Lambda)}(\tilde{\varphi}_j+\tilde{\varphi}_{j+1}),\ \forall 0\leq j\leq N.
\label{p2spectral3}\end{equation}

Replacing (\ref{p2spectral3}) into (\ref{p2spectral2}), we obtain that
\begin{equation}-\frac{1}{2}\tilde{\varphi}_{j-1}+\frac{3\Lambda^2-104\Lambda+240}{\Lambda^2+16\Lambda+240}\tilde{\varphi}_j
-\frac{1}{2}\tilde{\varphi}_{j+1}=0,\ \forall 1\leq j\leq N,\mbox{ with }\tilde{\varphi}_0=\tilde{\varphi}_{N+1}=0.\label{p2spectral4}\end{equation}

It is easy to check that $\tilde{\varphi}_j=\sin(k\pi x_j)$, with $1\leq k\leq N$, solves (\ref{p2spectral4}). Then
the normalized eigenvalues $\Lambda$ verify the identity
\begin{equation}\cos(k\pi h)=w(\Lambda), \mbox{ with }w(\Lambda)=\frac{3\Lambda^2-104\Lambda+240}{\Lambda^2+16\Lambda+240}.\label{p2spectral5}\end{equation}

For each $\eta\in[0,\pi]$, consider the following second-order algebraic equation in $\Lambda=\Lambda(\eta)$:
\begin{equation}(3-\cos(\eta))\Lambda^2-2\Lambda(52+8\cos(\eta))+240(1-\cos(\eta))=0,\label{p2spectral6}\end{equation}
whose solutions are
\begin{equation}\Lambda=\Lambda^{a}(\eta):=\frac{120\sin^2(\eta/2)}{11+4\cos^2(\eta/2)+
\sqrt{\Delta(\eta)}} \mbox{ and }\Lambda=\Lambda^{o}(\eta):=
\frac{22+8\cos^2(\eta/2)+2\sqrt{\Delta(\eta)}}{1+\sin^2(\eta/2)},\label{FourierSymbols}\end{equation}
with $\Delta(\eta):=1+268\cos^2(\eta/2)-44\cos^4(\eta/2)$. The superscripts $a$ and $o$ stand for \textit{acoustic} and \textit{optic}.
We will also need the square roots of the Fourier symbols (\ref{FourierSymbols}), the so called
\textit{dispersion relations}:
\begin{equation}\lambda^{a}(\eta):=\sqrt{\Lambda^{a}(\eta)}
\mbox{ and }\lambda^{o}(\eta):=\sqrt{\Lambda^{o}(\eta)}.\label{DispersionRelations}\end{equation}

For each $1\leq k\leq N$, set
\begin{equation}\Lambda^{a,k}:=\Lambda^{a}(k\pi h)\mbox{ and }\Lambda^{o,k}:=\Lambda^{o}(k\pi h).\label{p2eigenvalues}\end{equation}
We refer to $(\Lambda^{a,k})_{1\leq k\leq N}$ and $(\Lambda^{o,k})_{1\leq k\leq N}$ as the \textit{acoustic} and the \textit{optic branch} of the spectrum. The corresponding eigenvectors are
\begin{equation}\left\{\begin{array}{l}\tilde{\varphi}^{a,k}_j=\sin(k\pi x_j),\ 0\leq j\leq N+1, \mbox{ and }
\tilde{\varphi}^{a,k}_{j+1/2}=\frac{40+\Lambda^{a,k}}{4(10-\Lambda^{a,k})}\cos\big(\frac{k\pi h}{2}\big)\sin(k\pi x_{j+1/2}),\ 0\leq j\leq N,\smallskip\\
\tilde{\varphi}^{o,k}_j=\sin(k\pi x_j),\ 0\leq j\leq N+1, \mbox{ and }
\tilde{\varphi}^{o,k}_{j+1/2}=\frac{40+\Lambda^{o,k}}{4(10-\Lambda^{o,k})}\cos\big(\frac{k\pi h}{2}\big)\sin(k\pi x_{j+1/2}),\ 0\leq j\leq N.
\end{array}
\right.\label{p2eigenvectors}\end{equation}

\textbf{The resonant mode.} Up to this moment, we have explicitly calculated $2N$ solutions of the eigenvalue problem (\ref{p2spectral}). To do this, we have supposed that $\Lambda\not=10$. But $\Lambda=\Lambda^r:=10$ is also an eigenvalue corresponding to the \textit{resonant} mode, the superscript $r$ standing for
\textit{resonant}. The components of the corresponding eigenvector $\mathbf{\varphi}_h^r$ are
\begin{equation}\tilde{\varphi}_j^r=0,\ 0\leq j\leq N+1,\mbox{ and }\tilde{\varphi}_{j+1/2}^r=(-1)^j,\ 0\leq j\leq N.\label{p2resonanteigenvector}\end{equation}

For any normalized eigenvalue $\Lambda$, define $\Lambda_h:=\Lambda/h^2$, $\lambda:=\sqrt{\Lambda}$,
$\lambda_h:=\sqrt{\Lambda_h}$ (see Figure \ref{figp2eigenvalues}).
\begin{figure}\begin{center}\includegraphics[width=7.5cm,height=5cm]{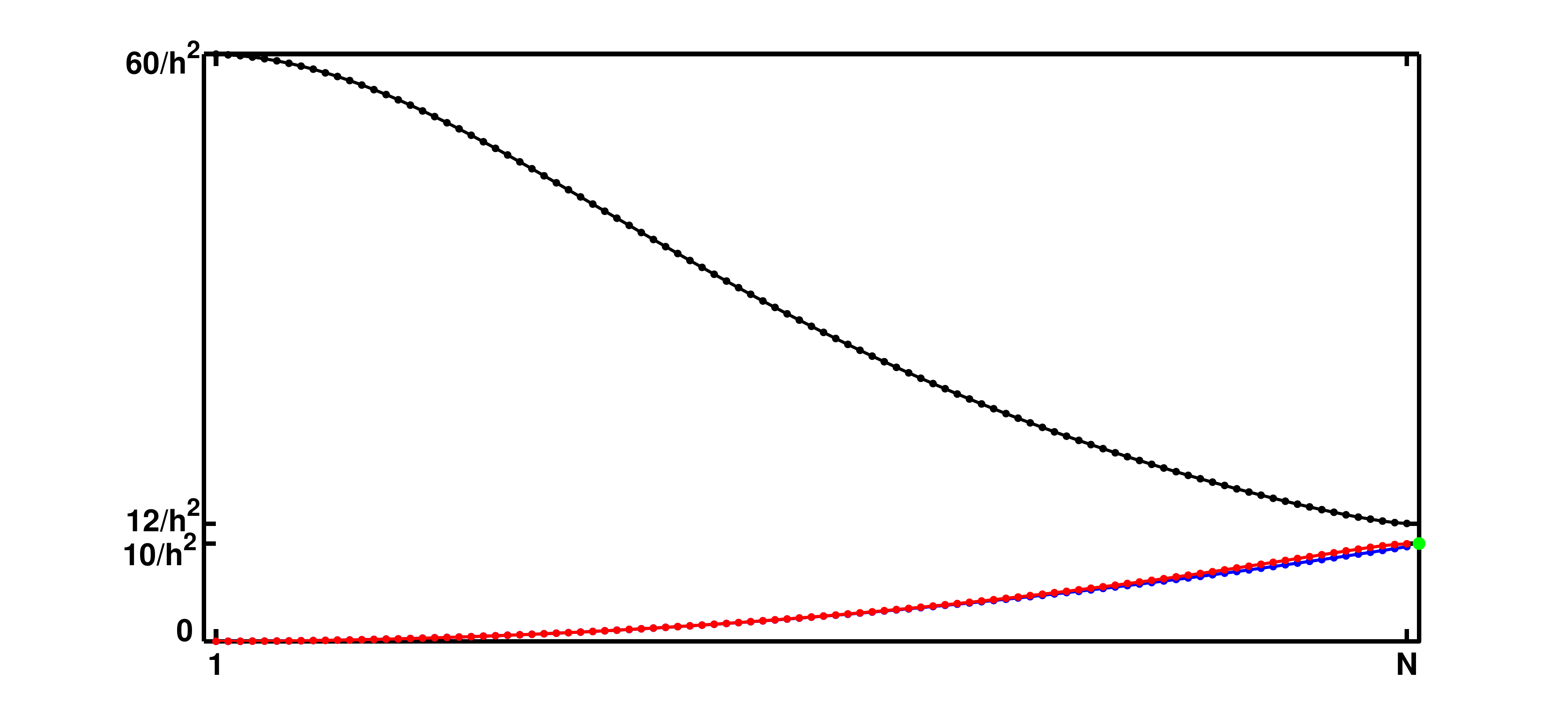}\includegraphics[width=7.5cm,height=5cm]{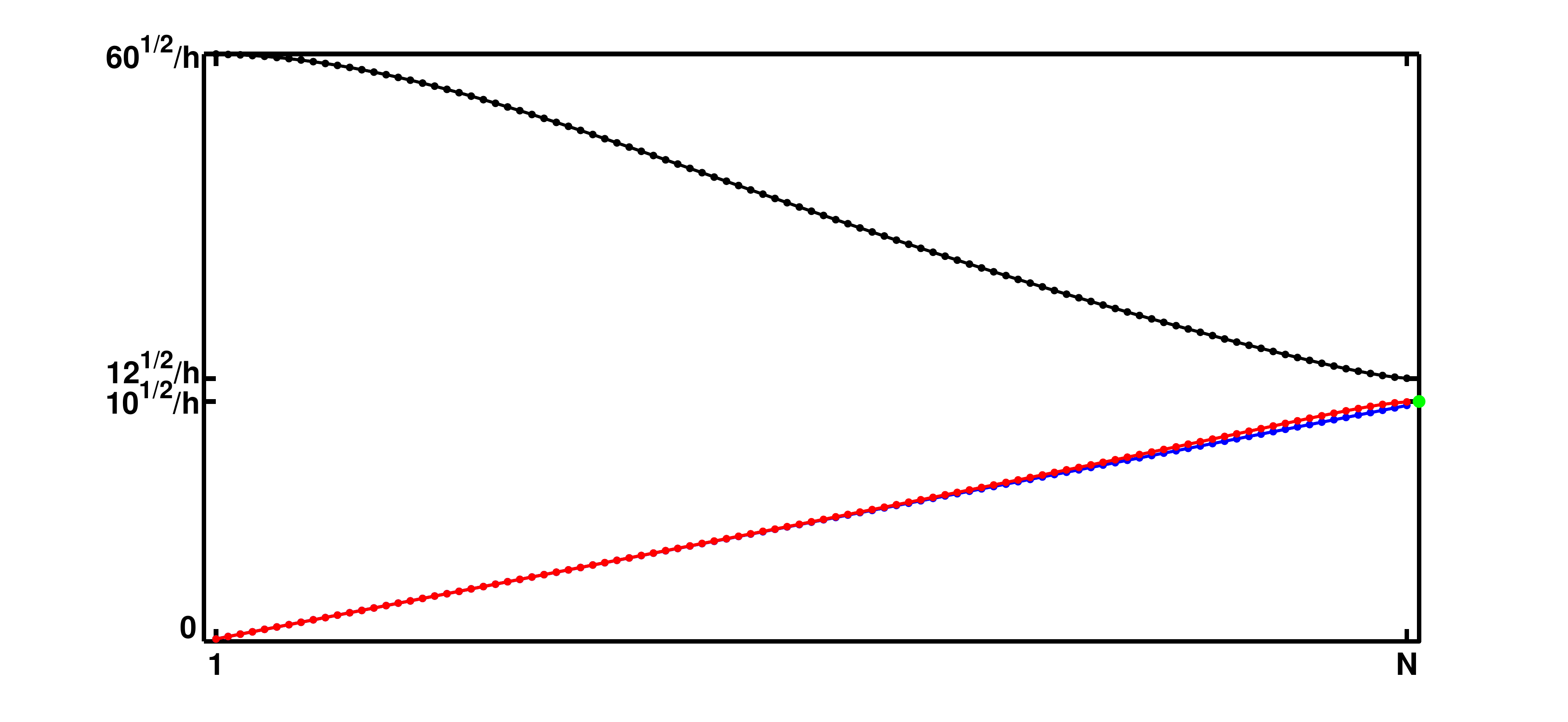}\end{center}
\caption{The eigenvalues $\Lambda_h$ (left) versus their square roots
$\lambda_h$ (right): the continuous ones (blue), the acoustic (red), the optic (black) and the resonant (green) modes. }\label{figp2eigenvalues}\end{figure}

\textbf{Normalized eigenvectors.} For any eigenvector $\mathbf{\tilde{\varphi}}_h\in\{\mathbf{\tilde{\varphi}}^{a,k}_h,\
\mathbf{\tilde{\varphi}}^{o,k}_,\ 1\leq k\leq N,\ \mathbf{\tilde{\varphi}}^r_h\}$, we define the $L^2$-normalized eigenvector
$\mathbf{\varphi}_h:=\mathbf{\tilde{\varphi}}_h/||\mathbf{\tilde{\varphi}}_h||_{h,0}$.

\begin{figure}
 \begin{center}\includegraphics[width=5.5cm,height=4.5cm]{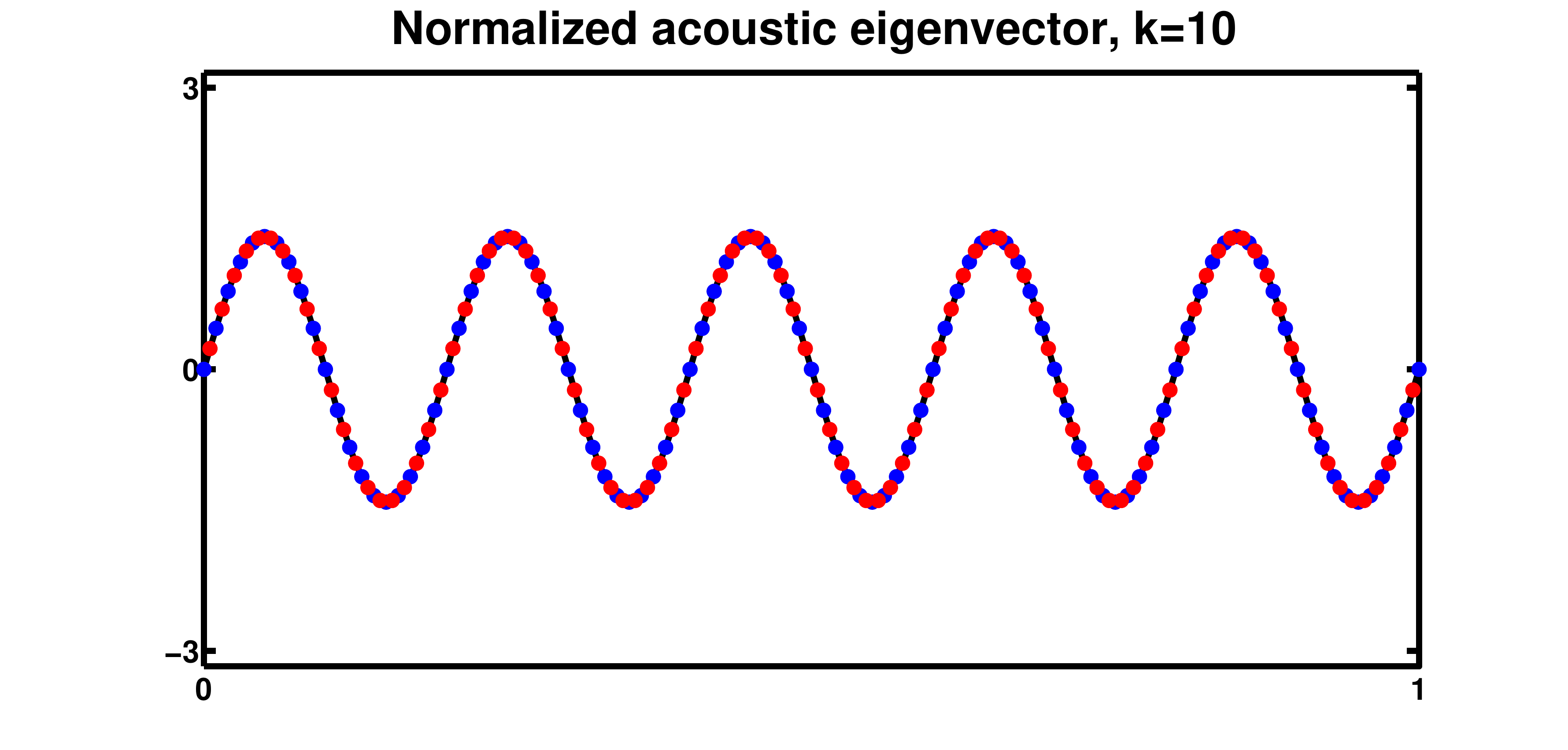}\includegraphics[width=5.5cm,height=4.5cm]{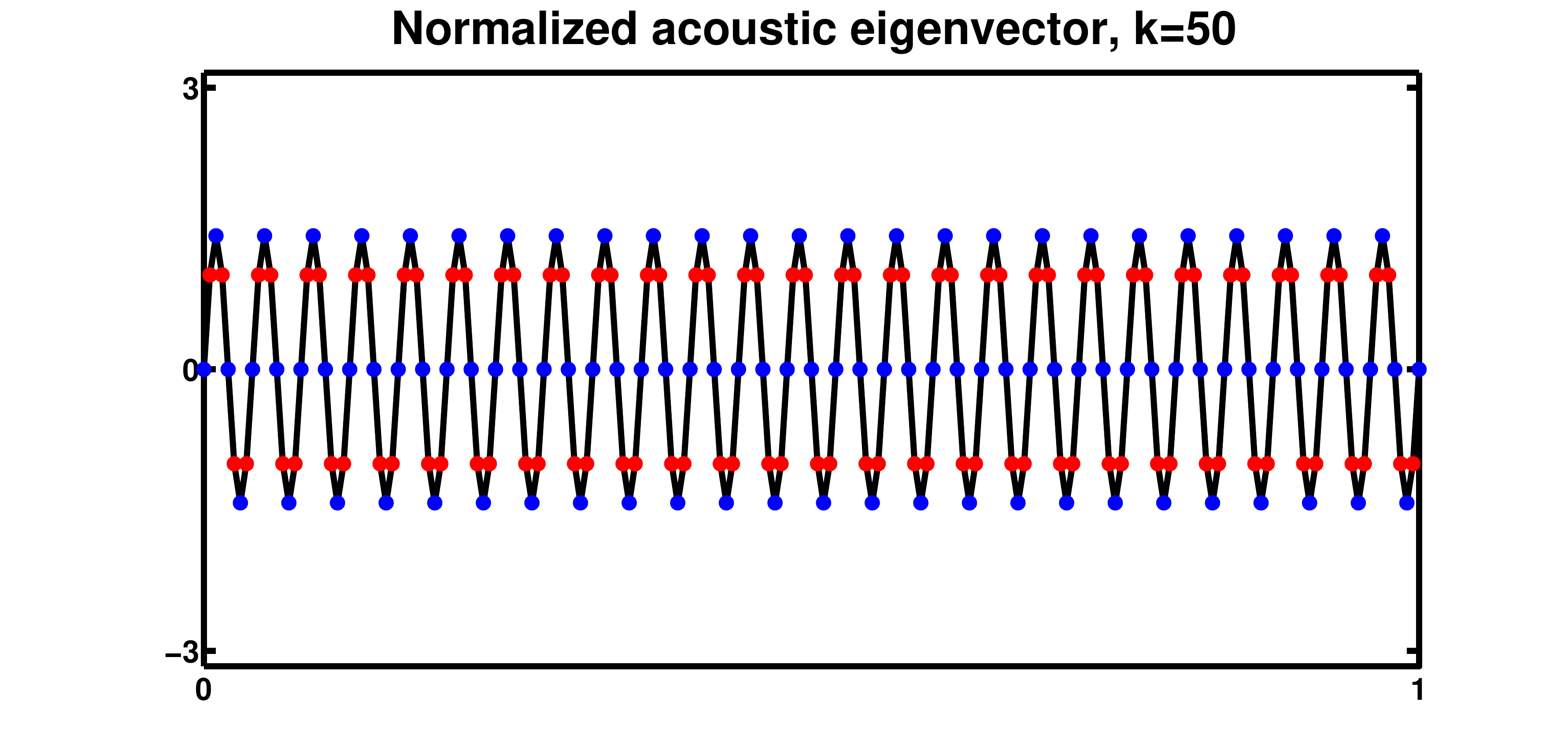}\includegraphics[width=5.5cm,height=4.5cm]{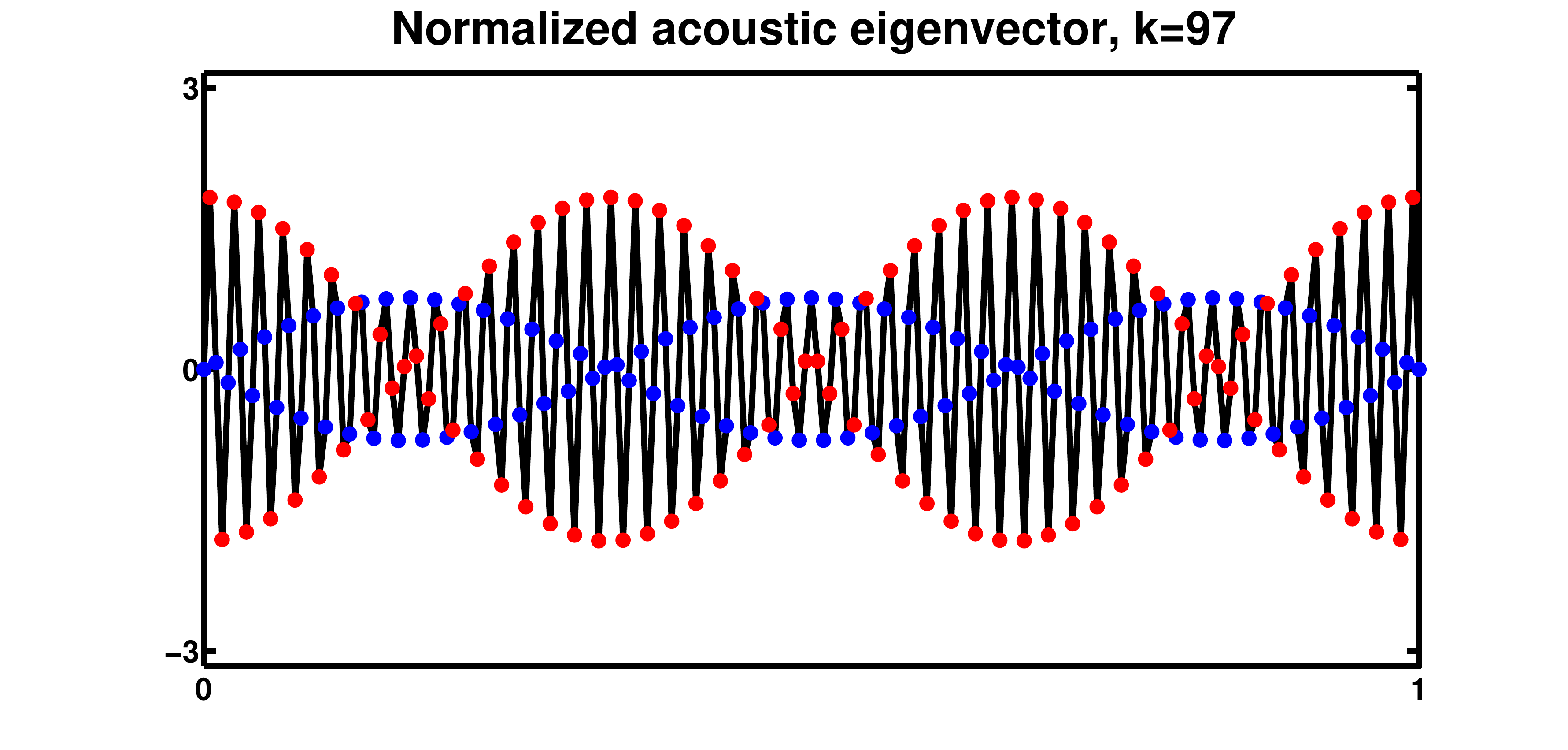}\\
 \includegraphics[width=5.5cm,height=4.5cm]{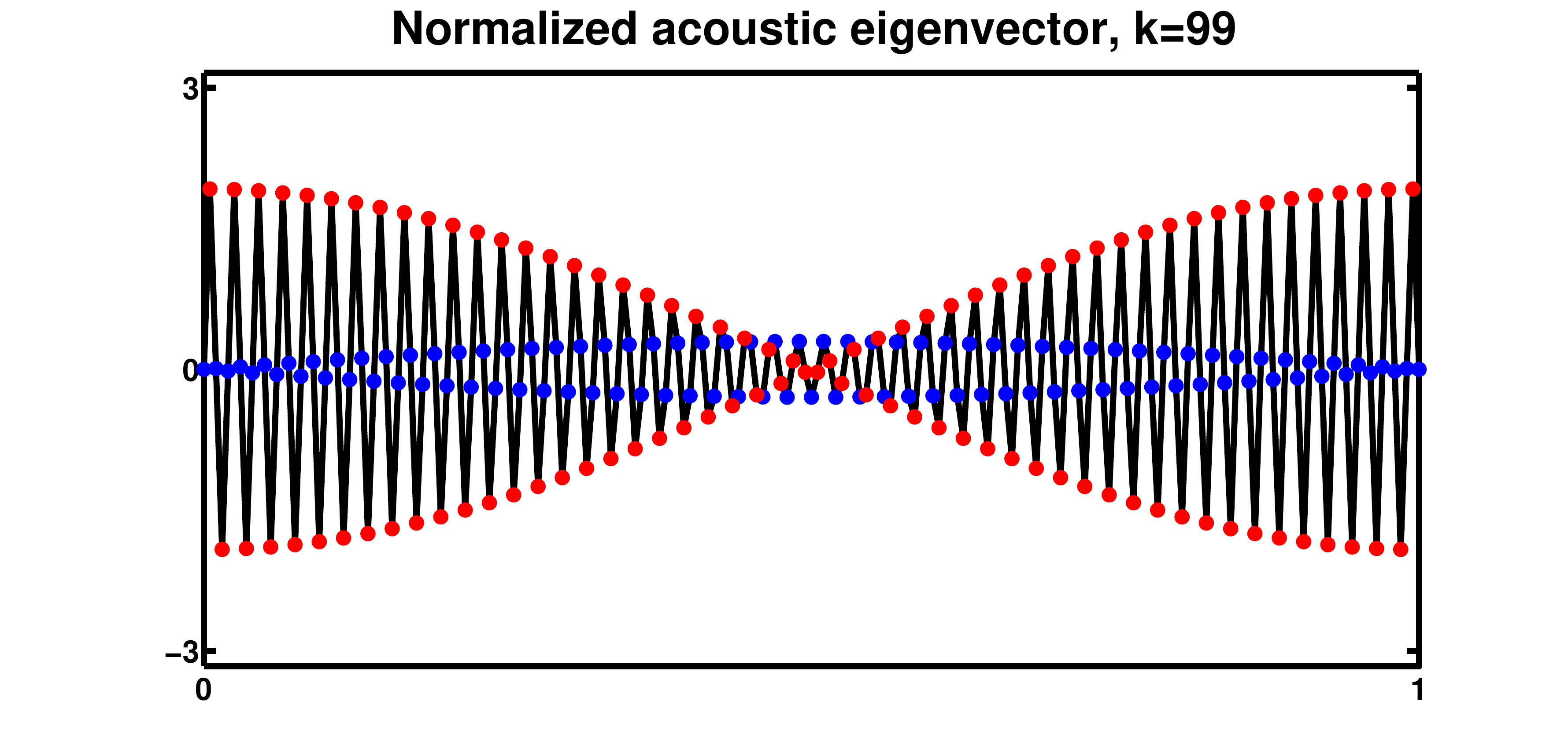}\includegraphics[width=5.5cm,height=4.5cm]{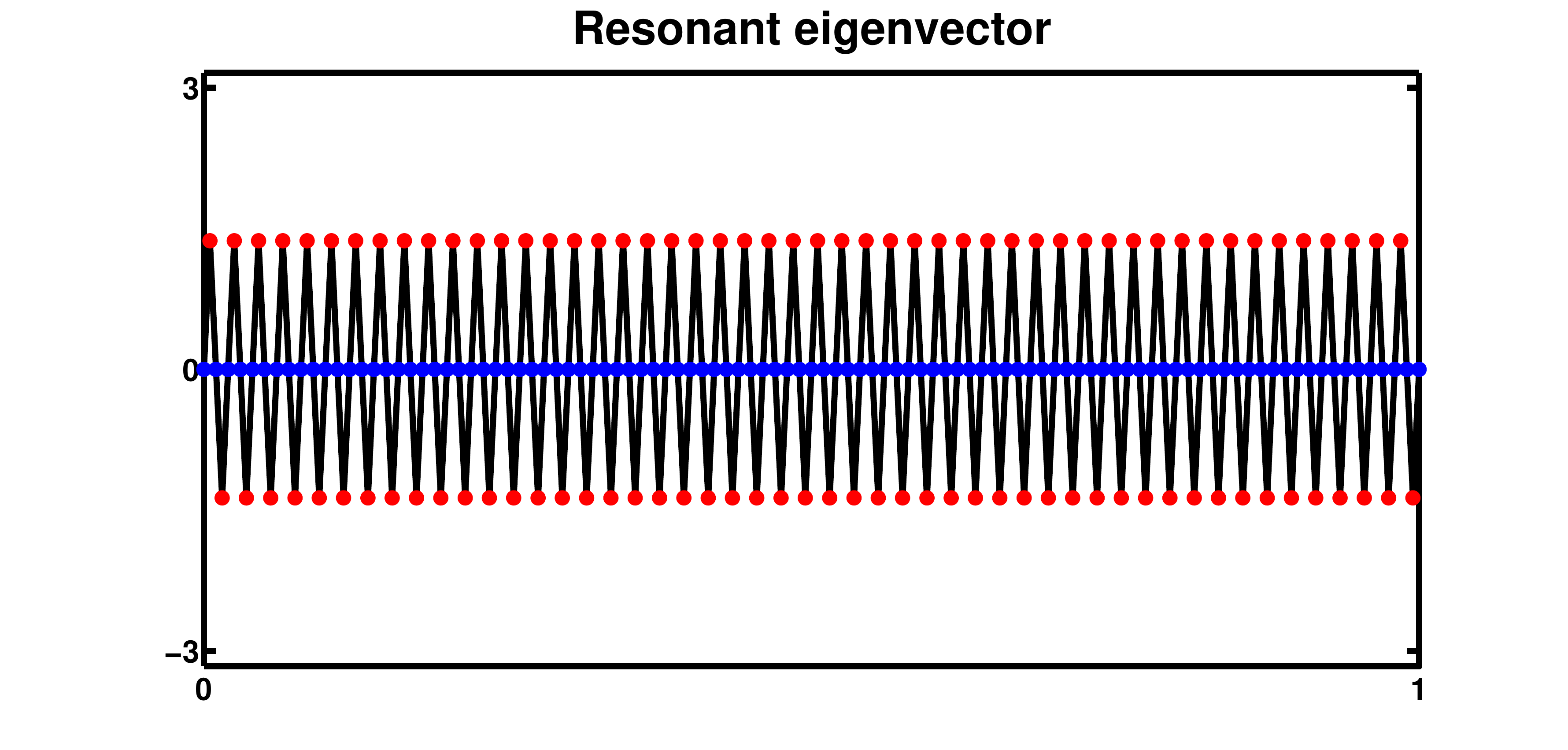}\includegraphics[width=5.5cm,height=4.5cm]{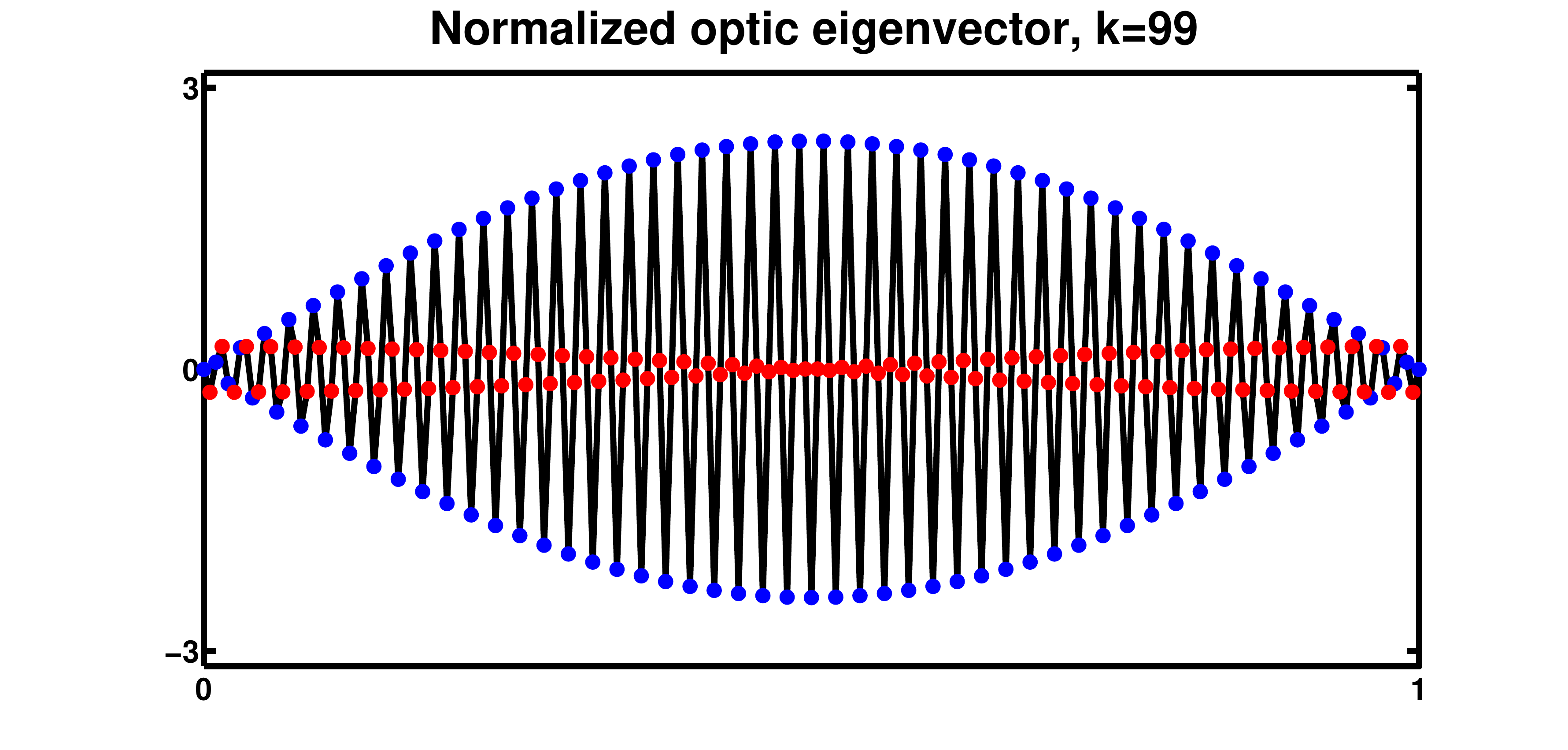}\\
 \includegraphics[width=5.5cm,height=4.5cm]{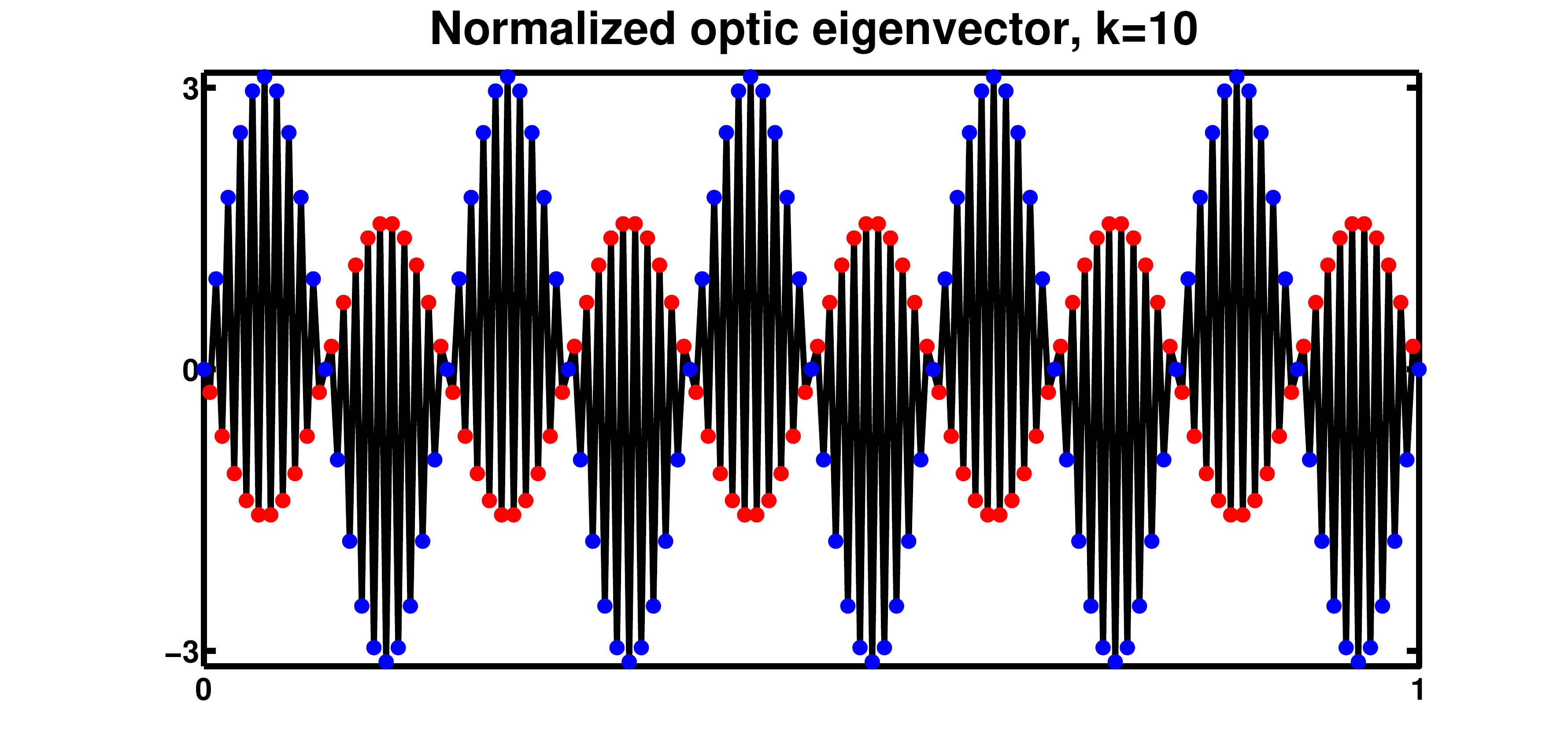}\includegraphics[width=5.5cm,height=4.5cm]{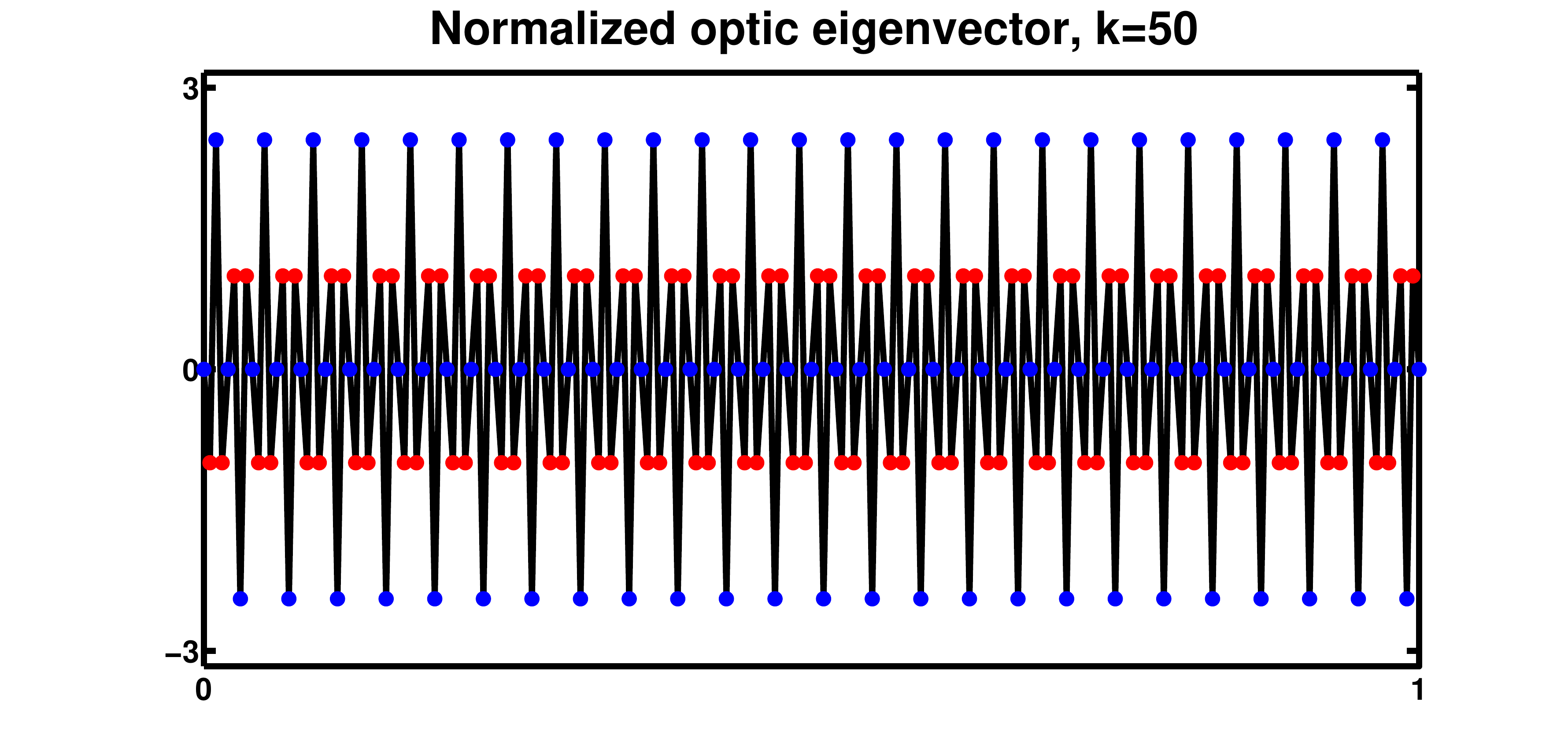}\includegraphics[width=5.5cm,height=4.5cm]{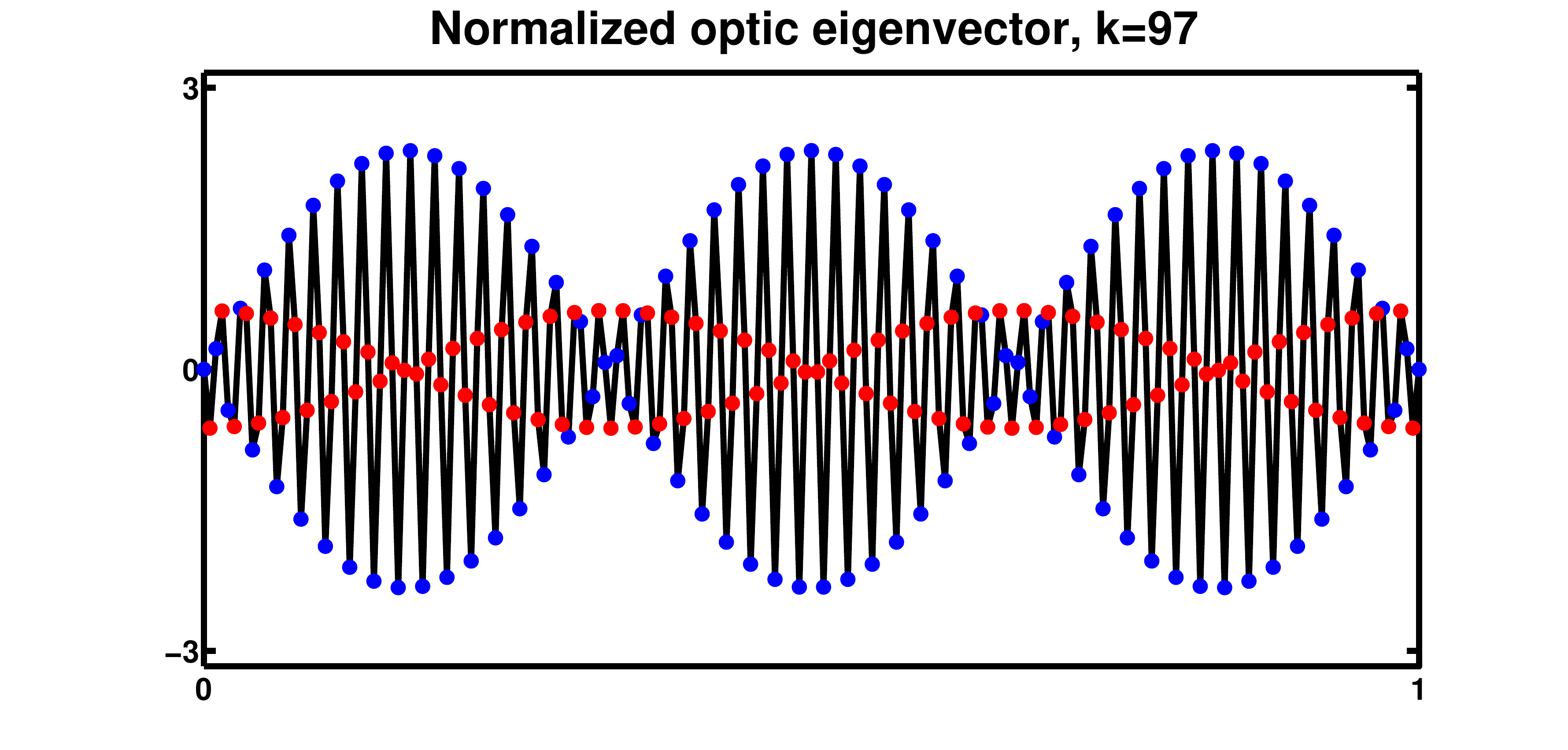}\end{center}
  \caption{Examples of normalized acoustic, optic and resonant eigenvectors for $N=99$. In blue/red, we represent their nodal/midpoint components.}\label{figeigenvectors}
\end{figure}

Using the expression (\ref{discretenorms}) of the discrete norm $||\cdot||_{h,0}$ and $||\cdot||_{h,1}$ and the identity (\ref{p2spectral3}), we obtain
the following representations of the $||\cdot||_{h,0}$ and $||\cdot||_{h,1}$-norms of the acoustic and optic eigenvectors
in terms of their nodal components, for all $\alpha\in\{a,0\}$ and $1\leq k\leq N$:
\begin{equation}||\mathbf{\tilde{\varphi}}^{\alpha,k}_h||_{h,0}^2=\frac{1}{24}\Big[1+\frac{500}{|10-\Lambda^{\alpha,k}|^2}\Big]h\sum\limits_{j=0}^N|\tilde{\varphi}_{j+1}^{\alpha,k}+\tilde{\varphi}_j^{\alpha,k}|^2+
\frac{h}{12}\sum\limits_{j=0}^N|\tilde{\varphi}_{j+1}^{\alpha,k}-\tilde{\varphi}_j^{\alpha,k}|^2\label{h0normeig}\end{equation}
and
\begin{equation}||\mathbf{\tilde{\varphi}}^{\alpha,k}_h||_{h,1}^2=
h\sum\limits_{j=0}^N\Big|\frac{\tilde{\varphi}^{\alpha,k}_{j+1}-\tilde{\varphi}^{\alpha,k}_j}{h}\Big|^2
+\frac{1}{h^2}\frac{4}{3}\Big|\frac{5\Lambda^{\alpha,k}}{4(10-\Lambda^{\alpha,k})}\Big|^2h\sum\limits_{j=0}^N|\tilde{\varphi}_{j+1}^{\alpha,k}+
\tilde{\varphi}_j^{\alpha,k}|^2.\label{h1normeig}\end{equation}

Then, using the representation formula (\ref{h0normeig}), the identities
\begin{equation}\begin{array}{c}h\sum\limits_{j=0}^N\sin^2(k\pi x_j)=\frac{1}{2},\quad
h\sum\limits_{j=0}^N|\sin(k\pi x_{j+1})-\sin(k\pi x_j)|^2=2\sin^2\big(\frac{k\pi h}{2}\big),\\
h\sum\limits_{j=0}^N|\sin(k\pi x_{j+1})+\sin(k\pi x_j)|^2=2\cos^2\big(\frac{k\pi h}{2}\big), \forall 1\leq k\leq N,
                \end{array}\label{TrigonometricIdentities}\end{equation}
and (\ref{p2spectral5}), we obtain
\begin{equation}||\mathbf{\tilde{\varphi}}^{\alpha,k}_h||_{h,0}^2=\frac{1}{3\tilde{W}(\Lambda^{\alpha,k})}, \mbox{ with }\tilde{W}(\Lambda)=
\frac{(\Lambda-10)(\Lambda^2+16\Lambda+240)}{19\Lambda^2+120\Lambda-3600}.\label{normh0}\end{equation}
Let us remark that $||\mathbf{\tilde{\varphi}}^{a,k}_h||_{h,0}$ blows-up as $kh\to1$. With the above notation,
\begin{equation}\varphi_j^{\alpha,k}=n^{\alpha,k}\sin(k\pi x_j)\mbox{ and }\varphi_{j+1/2}^{\alpha,k}=
m^{\alpha,k}\sin(k\pi x_{j+1/2}),\ \forall\alpha\in\{a,o\},\label{p2eignormalized}\end{equation} where
$$n^{\alpha,k}=\sqrt{3\tilde{W}(\Lambda^{\alpha,k})},
\quad m^{\alpha,k}=n^{\alpha,k}\frac{40+\Lambda^{\alpha,k}}{4(10-\Lambda^{\alpha,k})}\cos\big(\frac{k\pi h}{2}\big),$$
$n$ and $m$ standing for the \textit{nodal} and \textit{midpoint} components.

Using the explicit form of the $||\cdot||_{h,0}$ - norm and the characterization of the resonant mode (\ref{p2resonanteigenvector}),
we obtain that $||\mathbf{\tilde{\varphi}}^{r}_h||_{h,0}^2=8/15$ and therefore, the normalized resonant mode
$\mathbf{\varphi}^{r}$ satisfies
\begin{equation}\varphi_j^{r}=0,\ \forall 0\leq j\leq N+1 \mbox{ and } \varphi_{j+1/2}^r=(-1)^j\sqrt{15}/2\sqrt{2}, \ \forall 0\leq j\leq N.\label{p2eignormalizedresonant}\end{equation}

Let us introduce the sets of eigenvalues and of $\mathcal{H}_h^0$-normalized eigenfunctions, i.e.
\begin{equation}\mathcal{EV}_h:=\{\Lambda^{\alpha,k}_h,\alpha\in\{a,o\},1\leq k\leq N,\Lambda^r_h\} \mbox{ and }
\mathcal{EF}_h:=\{\mathbf{\varphi}^{\alpha,k}_h,\alpha\in\{a,o\},1\leq k\leq N,\mathbf{\varphi}^r_h\}.\label{p2SetEigenvectors}\end{equation}

\textbf{Fourier representation of discrete solutions.} Since $\mathcal{EF}_h$
is an orthonormal basis in $\mathcal{H}_h^0$, the initial data in (\ref{p2adjoint}) admit the following Fourier representation:
\begin{equation}\mathbf{U}_h^i=\sum\limits_{k=1}^N\widehat{u}^{a,k,i}\mathbf{\varphi}^{a,k}_h
+\widehat{u}^{r,i}\mathbf{\varphi}^r_h
+\sum\limits_{k=1}^N\widehat{u}^{o,k,i}\mathbf{\varphi}^{o,k}_h,\mbox{ where }
\widehat{u}^{\alpha,k,i}=(\mathbf{U}_h^i,\mathbf{\varphi}^{\alpha,k}_h)_{h,0} \mbox{ and }
\widehat{u}^{r,i}=(\mathbf{U}_h^i,\mathbf{\varphi}^r_h)_{h,0},\label{FourierRepresentationData}\end{equation}
with $\alpha\in\{a,o\}$, $1\leq k\leq N$ and $i=0,1$. Correspondingly, the solution of (\ref{p2adjoint}) can be represented as follows:
\begin{equation}\mathbf{U}_h(t)=\sum\limits_{\pm}\Big[\sum\limits_{k=1}^N\widehat{u}^{a,k}_{\pm}\exp(\pm it\lambda_h^{a,k})
\mathbf{\varphi}^{a,k}_h+
\widehat{u}^r_{\pm}\exp(\pm it\lambda_h^r)\mathbf{\varphi}^r_h+
\sum\limits_{k=1}^N\widehat{u}^{o,k}_{\pm}\exp(\pm it\lambda_h^{o,k})\mathbf{\varphi}^{o,k}_h\Big],
\label{p2adjointSol}\end{equation}
where $$\widehat{u}^{\alpha,k}_{\pm}=\frac{1}{2}\Big(\widehat{u}^{\alpha,k,0}\pm\frac{\widehat{u}^{\alpha,k,1}}{i\lambda_h^{\alpha,k}}\Big),
\forall\alpha\in\{a,o\},\ 1\leq k\leq N,\quad \widehat{u}^r_{\pm}=\frac{1}{2}\Big(\widehat{u}^{r,0}\pm\frac{\widehat{u}^{r,1}}{i\lambda_h^r}\Big).$$

The total energy (\ref{p2energy}) of the solutions of (\ref{p2adjoint}) is then as follows:
\begin{equation}\mathcal{E}_h(\mathbf{U}_h^0,\mathbf{U}_h^1)=
\sum\limits_{k=1}^N\Lambda_h^{a,k}(|\widehat{u}^{a,k}_+|^2+|\widehat{u}^{a,k}_-|^2)+\Lambda_h^r(|\widehat{u}^r_+|^2+|\widehat{u}^r_-|^2)
+\sum\limits_{k=1}^N\Lambda_h^{o,k}(|\widehat{u}^{o,k}_+|^2+|\widehat{u}^{o,k}_-|^2).\label{p2energyFourier}\end{equation}

\textbf{Limits and monotonicity of the eigenvalues.} Firstly, let us remark that as $kh\to1$, $\Lambda^{a,k}\to 10$, $\Lambda^{o,k}\to 12$ and as $kh\to 0$, $\Lambda^{o,k}\to 60$.
On the other hand, the so-called \textit{group velocities}, which are first-order derivatives of the Fourier symbols (\ref{FourierSymbols})
or of the corresponding dispersion relations (\ref{DispersionRelations}), verify the following positivity condition
$$\partial_{\eta}\lambda^{a}(\eta),\ \partial_{\eta}\Lambda^{a}(\eta),\ -\partial_{\eta}\lambda^{o}(\eta),\ -\partial_{\eta}\Lambda^{o}(\eta)>0,\ \forall\eta\in(0,\pi),$$
which means that \textit{the acoustic branch is strictly increasing} and \textit{the optic one is strictly decreasing} in $k$. 
Consequently, the high frequency wave packets involving only the acoustic or the optic modes and concentrated around a given frequency $1\leq k^{\star}\leq N$ 
propagate in opposite directions. Moreover, at $\eta=0$ or
$\eta=\pi$, the group velocities satisfy
$$\partial_{\eta}\Lambda^{a}(\pi)=\partial_{\eta}\Lambda^{o}(\pi)=\partial_{\eta}\Lambda^{a}(0)=\partial_{\eta}\Lambda^{o}(0)=0,\quad\partial_{\eta}\lambda^{a}(\pi)=
\partial_{\eta}\lambda^{o}(\pi)=\partial_{\eta}\lambda^{o}(0)=0\quad\mbox{and}\quad\partial_{\eta}\lambda^{a}(0)=1,$$
which, according to the analysis in \cite{MarZuaDG}, shows, in particular, that there are waves concentrated on each mode which propagate at arbitrarily slow velocity.

\section{Boundary observability of eigenvectors}\label{SectBoundObsEig}
The main result of this section is as follows:
\begin{proposition} For all $\alpha\in\{a,o\}$ and all $1\leq k\leq N$, the following identity holds for both acoustic and optic eigensolutions:
\begin{equation}||\mathbf{\varphi}^{\alpha,k}_h||_{h,1}^2=\frac{1}{W(\Lambda^{\alpha,k})}
\Big|\frac{\varphi_N^{\alpha,k}}{h}\Big|^2, \mbox{ with }
W(\Lambda)=\frac{24(\Lambda-10)^2(\Lambda-12)(\Lambda-60)}{(-19\Lambda^2-120\Lambda+3600)(\Lambda^2+16\Lambda+240)}.
\label{ObservabilityEigenvectorsao}\end{equation}
Moreover, for the resonant mode, the following identity holds:
\begin{equation}||\mathbf{\varphi}^{r}_h||_{h,1}^2=\frac{16}{3}\Big|\frac{\varphi_{N+1/2}^{r}}{h}\Big|^2.\label{ObservabilityEigenvectorsr}\end{equation}
\label{PropObservabilityEigenvectors}
\end{proposition}

\begin{remark}The identity (\ref{ObservabilityEigenvectorsao}) is the discrete analogue of the continuous one $||\varphi^k||_{H_0^1}^2=|\varphi_x^k(1)|^2/2$, where $\varphi^k(x)=\sqrt{2}\sin(k\pi x)$ is the $L^2$-normalized eigenfunction corresponding to the eigenvalue $\Lambda^k=k^2\pi^2$.\label{remark0}\end{remark}

\begin{remark}Due to the monotonicity of the Fourier symbols, we have that $\Lambda^{a,k}\in(0,10)$ and
$\Lambda^{o,k}\in(12,60)$, for all $1\leq k\leq N$. The quadratic equation $-19x^2-120x+3600=0$ has the roots $x_1=-60(1+\sqrt{20})/19<0$ and
$x_2=60/(1+\sqrt{20})\in(10,12)$. This allows us to guarantee that $W(\Lambda)>0$, for all $\Lambda\in(0,10)\cup(12,60)$.\label{remark1}\end{remark}

\begin{remark}Due to the form of the denominator in the right hand side of (\ref{ObservabilityEigenvectorsao}) and from the above lower and upper bound of the Fourier symbols and the behavior
of the group velocities, we deduce that the coefficient of $|\varphi_N^{\alpha,k}/h|^2$ in the right hand side of (\ref{ObservabilityEigenvectorsao})
is singular as  $kh\to 1$ both when $\alpha=a$ or $\alpha=o$ and when  $kh\to0$ and $\alpha=o$.
\label{remark2}\end{remark}

\begin{proof}[Proof of Proposition \ref{PropObservabilityEigenvectors}] Fix $\alpha\in\{a,o\}$ and $1\leq k\leq N$. Obviously,
it is enough to prove (\ref{ObservabilityEigenvectorsao}) for the un-normalized eigenvectors $\mathbf{\tilde{\varphi}}^{\alpha,k}$.
We will use two approaches to prove the identity (\ref{ObservabilityEigenvectorsao}). The first one consists on using the classical multiplier $x_j(\tilde{\varphi}_{j+1}-\tilde{\varphi}_{j-1})/2h$ (which is a discrete version of the continuous one $x\varphi_x$)
in the simplified spectral problem (\ref{p2spectral6}) and then to apply
the \textit{Abel summation by parts formula}
\begin{equation}\sum\limits_{j=1}^N(a_{j+1}-a_j)b_j=a_{N+1}b_{N+1}-a_1b_0-\sum\limits_{j=0}^Na_{j+1}(b_{j+1}-b_j),\label{AbelFormula}\end{equation}
for all $(a_j)_{1\leq j\leq N+1}\in \cc^{N+1}$ and
$(b_j)_{0\leq j\leq N+1}\in\cc^{N+2}$. In what follows, we will add the superscript $\alpha,k$ to the solution $\mathbf{\tilde{\varphi}}_h$ of (\ref{p2spectral4}). In this way, we deduce the following identity:
$$\frac{h}{2}\sum\limits_{j=0}^N\big|\tilde{\varphi}_{j+1}^{\alpha,k}\pm\tilde{\varphi}_j^{\alpha,k}\big|^2=\frac{1}{2}\big|\tilde{\varphi}_N^{\alpha,k}\big|^2
+(w(\Lambda^{\alpha,k})\pm1)
h\sum\limits_{j=0}^N\tilde{\varphi}_{j+1}^{\alpha,k}\tilde{\varphi}_j^{\alpha,k}.$$
Replacing the crossed sum $h\sum_{j=0}^N\tilde{\varphi}_{j+1}^{\alpha,k}\tilde{\varphi}_j^{\alpha,k}$ obtained from the identity with $+$ into the one with
$-$, we get the following equality:
\begin{equation}\label{ObsEig}h\sum\limits_{j=0}^N\Big|\frac{\tilde{\varphi}_{j+1}^{\alpha,k}-\tilde{\varphi}_j^{\alpha,k}}{h}\Big|^2=\frac{2}{w(\Lambda^{\alpha,k})+1}
\Big|\frac{\tilde{\varphi}_N^{\alpha,k}}{h}\Big|^2+\frac{1}{h^2}\frac{w(\Lambda^{\alpha,k})-1}{w(\Lambda^{\alpha,k})+1}
h\sum\limits_{j=0}^N\big|\tilde{\varphi}_{j+1}^{\alpha,k}+\tilde{\varphi}_j^{\alpha,k}\big|^2.\end{equation}

Using the representation (\ref{h1normeig}) of the $||\cdot||_{h,1}$-norm of the optic and acoustic eigenvectors in terms of the nodal components,
we obtain that (\ref{ObsEig}) is equivalent to
\begin{equation}\label{ObsEig1}||\mathbf{\tilde{\varphi}}^{\alpha,k}_h||_{h,1}^2=
\frac{2}{w(\Lambda^{\alpha,k})+1}\Big|\frac{\tilde{\varphi}_N^{\alpha,k}}{h}\Big|^2
+\frac{1}{h^2}\Big(\frac{w(\Lambda^{\alpha,k})-1}{w(\Lambda^{\alpha,k})+1}+\frac{4}{3}\Big|\frac{5\Lambda^{\alpha,k}}{4(10-\Lambda^{\alpha,k})}\Big|^2\Big)
h\sum\limits_{j=0}^N\big|\tilde{\varphi}_{j+1}^{\alpha,k}+\tilde{\varphi}_j^{\alpha,k}\big|^2.\end{equation}

Replacing the representation of the $||\cdot||_{h,0}$-norm of the eigenvectors (\ref{h0normeig}) into the one of the
$||\cdot||_{h,1}$-norm, (\ref{h1normeig}), we obtain:
\begin{equation}||\mathbf{\tilde{\varphi}}^{\alpha,k}_h||_{h,1}^2=\frac{12}{h^2}||\mathbf{\tilde{\varphi}}^{\alpha,k}||_{h,0}^2
+\frac{1}{h^2}\frac{25\Lambda^{\alpha,k}(\Lambda^{\alpha,k}-12)-6(\Lambda^{\alpha,k}-10)(\Lambda^{\alpha,k}-60)}{12(\Lambda^{\alpha,k}-10)^2}
h\sum\limits_{j=0}^{N}\big|\tilde{\varphi}_{j+1}^{\alpha,k}+\tilde{\varphi}_j^{\alpha,k}\big|^2.\label{ObsEig2}\end{equation}
On the other hand, the $||\cdot||_{h,0}$ and $||\cdot||_{h,1}$-norms of the eigenvetors are related as follows:
\begin{equation}||\mathbf{\tilde{\varphi}}^{\alpha,k}_h||_{h,1}^2=
\frac{\Lambda^{\alpha,k}}{h^2}||\mathbf{\tilde{\varphi}}^{\alpha,k}_h||_{h,0}^2.\label{h0h1normeig}\end{equation}
Replacing (\ref{h0h1normeig}) into (\ref{ObsEig2}), we get
\begin{equation}h\sum\limits_{j=0}^N\big|\tilde{\varphi}_{j+1}^{\alpha,k}+\tilde{\varphi}_j^{\alpha,k}\big|^2
=\frac{12(\Lambda^{\alpha,k}-10)^2(\Lambda^{\alpha,k}-12)}{\Lambda^{\alpha,k}[-6(\Lambda^{\alpha,k}-60)(\Lambda^{\alpha,k}-10)+25\Lambda^{\alpha,k}(\Lambda^{\alpha,k}-12)]}
h^2||\mathbf{\tilde{\varphi}}^{\alpha,k}_h||_{h,1}^2.\label{ObsEig3}\end{equation}
By combining (\ref{ObsEig1}) and (\ref{ObsEig3}), we obtain
$$\Big[1-\frac{6(\Lambda^{\alpha,k}-60)(\Lambda^{\alpha,k}-10)+25\Lambda^{\alpha,k}(\Lambda^{\alpha,k}-12)}{-6(\Lambda^{\alpha,k}-60)(\Lambda^{\alpha,k}-10)+25\Lambda^{\alpha,k}(\Lambda^{\alpha,k}-12)}\Big]
||\mathbf{\tilde{\varphi}}^{\alpha,k}_h||_{h,1}^2=\frac{(\Lambda^{\alpha,k})^2+16\Lambda^{\alpha,k}+240}{2(\Lambda^{\alpha,k}-10)(\Lambda^{\alpha,k}-12)}\Big|\frac{\tilde{\varphi}_N^{\alpha,k}}{h}\Big|^2,$$
from which the identity (\ref{ObservabilityEigenvectorsao}) follows immediately.

The second approach to prove (\ref{ObservabilityEigenvectorsao}) is much more direct. It consists in using the representation (\ref{h1normeig}) of the $||\cdot||_{h,1}$-norm
of the eigenvectors, the trigonometric identities (\ref{TrigonometricIdentities}), the fact that $|\tilde{\varphi}_N^{\alpha,k}|=|\sin(k\pi h)|$ and the relation (\ref{p2spectral5}). Thus, for $W$ as in (\ref{ObservabilityEigenvectorsao}), we get:
$$\frac{||\mathbf{\tilde{\varphi}}^{\alpha,k}_h||_{h,1}^2}{\big|\frac{\tilde{\varphi}_N^{\alpha,k}}{h}\big|^2}
=\frac{1-w(\Lambda^{\alpha,k})+\frac{4}{3}\Big|\frac{5\Lambda^{\alpha,k}}{4(10-\Lambda^{\alpha,k})}\Big|^2(1+
w(\Lambda^{\alpha,k}))}
{(1-w(\Lambda^{\alpha,k}))(1+w(\Lambda^{\alpha,k}))}=\frac{1}{W(\Lambda^{\alpha,k})}.$$

The identity (\ref{ObservabilityEigenvectorsr}) follows by combining the explicit expressions of the components of the resonant eigenvector  (\ref{p2resonanteigenvector}) and  (\ref{discretenorms}). This concludes the proof of (\ref{ObservabilityEigenvectorsr}).
\end{proof}

\section{Discrete observability inequality: an Ingham approach}\label{SectIngham}
\textbf{The observability inequality.} In this section we prove that the discrete observability inequality (\ref{p2ObservabilityInequalityGeneral})
holds uniformly as $h\to 0$ in a truncated class of initial data for the observation operator $B_h$ introduced in (\ref{ObservationOperators}). More precisely, consider $0<\Lambda_+^a<10$ and $12<\Lambda_-^o\leq \Lambda_+^o<60$ and
correspondingly the wave numbers
\begin{equation}\eta_+^a:=(\Lambda^a)^{-1}(\Lambda^a_+) \mbox{ and }\eta^o_{\pm}:=(\Lambda^o)^{-1}(\Lambda^o_{\pm})\label{wavenumbers}\end{equation}
and introduce the subspace of $\cc^{2N+1}$ given by
$$\mathcal{T}_{h,\eta^a_+,\eta^o_-,\eta^o_+}:=\mbox{span}\{\mathbf{\varphi}^{a,k},k\pi h\leq \eta^a_+\}\oplus
\mbox{span}\{\mathbf{\varphi}^{o,k},\eta_+^o\leq k\pi h\leq \eta_-^o\}.$$
Consider the truncated subspace $\mathcal{S}_h\subset\mathcal{V}_h$ defined by (see Figure \ref{figp2TruncatedClass}):
\begin{equation}\mathcal{S}_h:=(\mathcal{T}_{h,\eta^a_+,\eta^o_-,\eta^o_+}\times \mathcal{T}_{h,\eta^a_+,\eta^o_-,\eta^o_+})\cap\mathcal{V}_h.\label{p2SubspTruncation}\end{equation}
\begin{figure}\begin{center}\includegraphics[width=9cm,height=4.5cm]{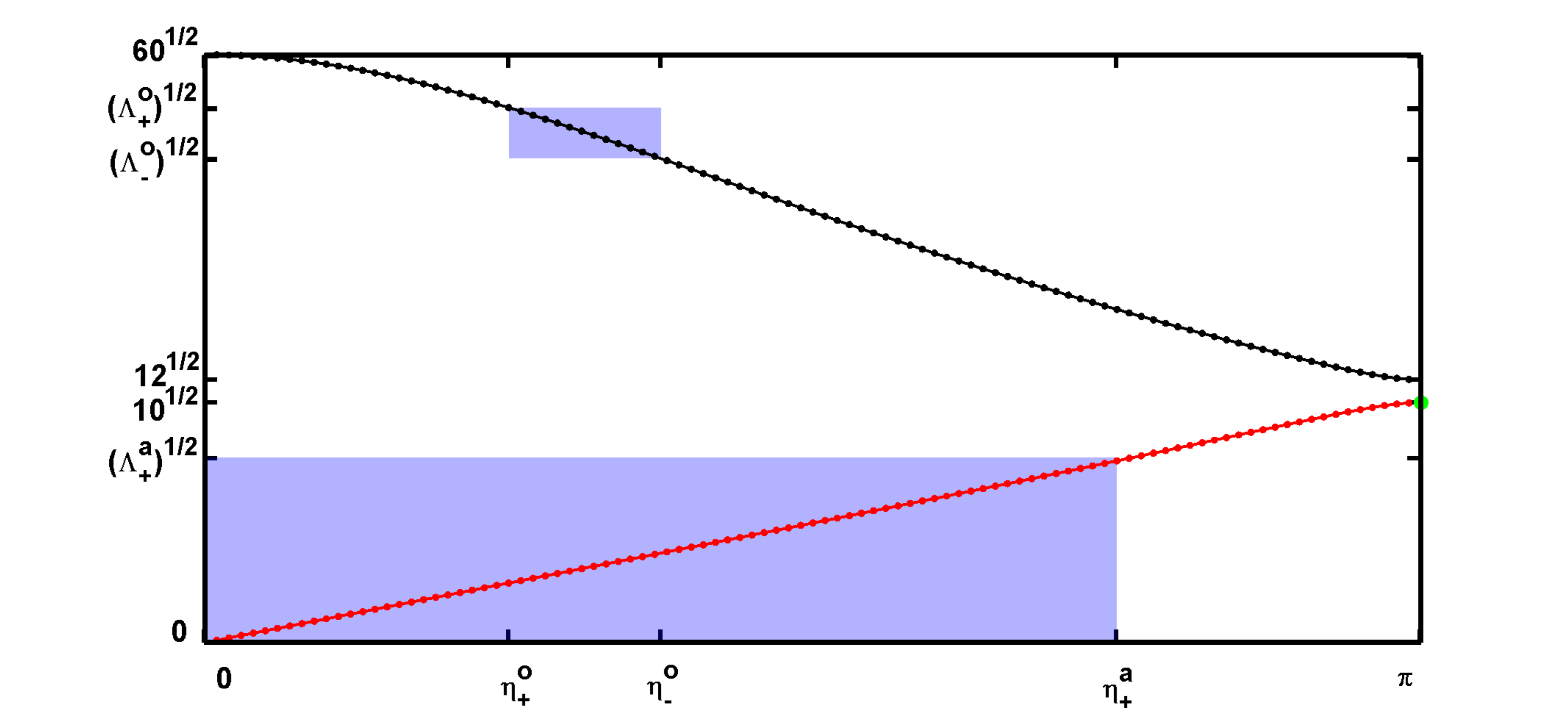}\end{center}
\caption{The selected area contains square roots of eigenvalues whose eigenvectors generate the truncated class $\mathcal{T}_{h,\eta_+^{a},\eta_-^{o},\eta_+^{o}}$. In red/black/green, the acoustic/optic/resonant mode. }\label{figp2TruncatedClass}\end{figure}
The main result of this subsection is as follows:
\begin{theorem}For all $\Lambda_+^a\in(0,10)$ and $\Lambda_-^o\leq\Lambda_+^o\in(12,60)$ independent of $h$, all initial data
$(\mathbf{U}^{h,0},\mathbf{U}^{h,1})\in\mathcal{S}_h$ introduced in (\ref{p2SubspTruncation}), with $\eta_+^a$, $\eta_{\pm}^o$ given by (\ref{wavenumbers}), and all
$$T>T^{\star}_{\eta_+^a,\eta_-^o,\eta_-^a}=\frac{2}{\min\{\min\limits_{\eta\in[0,\eta_+^a]}\partial_{\eta}\lambda^a,\min\limits_{\eta\in[\eta_+^o,\eta_-^o]}
(-\partial_{\eta}\lambda^o(\eta))\}}$$\label{TheoremFourierTruncation}
the observability inequality (\ref{p2ObservabilityInequalityGeneral}) holds uniformly as $h\to 0$ for the operator $B_h$ in (\ref{ObservationOperators}). \end{theorem}
\begin{proof}[Proof of Theorem \ref{TheoremFourierTruncation}] The fact that the eigenmodes involved in the class
$\mathcal{T}_{h,\eta_+^a,\eta_-^o,\eta_+^o}$ are such that the gap in each branch has a strictly positive lower bound uniformly as $h\to 0$ allows us
to apply Ingham Theorem (cf. \cite{LorKom}, Theorem 4.3, pp. 59).  More precisely, the spectral gap on each branch is bounded as follows:
$$\lambda_h^{a,k+1}-\lambda_h^{a,k}\geq \pi\min\limits_{\eta\in[0,\eta_+^a]}\partial_{\eta}\lambda^a(\eta)>0,\ \forall k\pi h,(k+1)\pi h\in[0,\eta_+^a]$$
and
$$\lambda_h^{o,k}-\lambda_h^{o,k+1}\geq \pi\min\limits_{\eta\in[\eta_+^o,\eta_-^o]}(-\partial_{\eta}\lambda^{o}(\eta))>0,\
\forall k\pi h,(k+1)\pi h\in[\eta_+^o,\eta_-^o].$$

Then $\gamma=\gamma(\eta_+^a,\eta_-^o,\eta_+^o):=\pi\min\{\min\limits_{\eta\in[0,\eta_+^a]}\partial_{\eta}\lambda^a(\eta),
\min\limits_{\eta\in[\eta_+^o,\eta_-^o]}(-\partial_{\eta}\lambda^{o}(\eta))\}>0$ is the uniform gap needed to apply the Ingham theory.
From (\ref{p2adjointSol}) and the definition of the class $\mathcal{T}_{h,\eta_+^a,\eta_-^o,\eta_+^o}$, we have
$$U_N(t)=\sum\limits_{\pm}\Big[\sum\limits_{k\pi h\in[0,\eta_+^a]}\widehat{u}^{a,k}_{\pm}\exp(\pm it\lambda_h^{a,k})\varphi^{a,k}_N
+\sum\limits_{k\pi h\in[\eta_+^o,\eta_-^o]}\widehat{u}^{o,k}_{\pm}\exp(\pm it\lambda_h^{o,k})\varphi^{o,k}_N\Big].$$
By applying the\textit{ inverse inequality} in Ingham Theorem (cf. \cite{LorKom}, pp. 60, (4.9)), we can guarantee that for all $T>2\pi/\gamma=T^{\star}_{\eta_+^a,\eta_-^o,\eta_+^o}$, there exists
a constant $C_-(T)>0$ independent of $h$ such that
$$C_-(T)\sum\limits_{\pm}\Big[\sum\limits_{k\pi h\in[0,\eta_+^a]}|\widehat{u}^{a,k}_{\pm}|^2\Big|\frac{\varphi_N^{a,k}}{h}\Big|^2
+\sum\limits_{k\pi h\in[\eta_+^o,\eta_-^o]}|\widehat{u}^{o,k}_{\pm}|^2\Big|\frac{\varphi_N^{o,k}}{h}\Big|^2
\Big]\leq \int\limits_0^T\Big|\frac{U_N(t)}{h}\Big|^2\,dt.$$

Using the identities (\ref{ObservabilityEigenvectorsao}) and 
$||\mathbf{\varphi}^{\alpha,k}_h||_{h,1}^2=\Lambda_h^{\alpha,k}$, we obtain that (the function $W$ below is as in (\ref{ObservabilityEigenvectorsao})):
$$C_-(T)\sum\limits_{\pm}\Big[\sum\limits_{k\pi h\in[0,\eta_+^a]}\Lambda_h^{a,k}|\widehat{u}^{a,k}_{\pm}|^2W(\Lambda^{a,k})
+\sum\limits_{k\pi h\in[\eta_+^o,\eta_-^o]}\Lambda^{o,k}_h|\widehat{u}^{o,k}_{\pm}|^2W(\Lambda^{o,k})
\Big]\leq \int\limits_0^T\Big|\frac{U_N(t)}{h}\Big|^2\,dt.$$

Taking into account that for our filtering algorithm $\Lambda\in[0,\Lambda_+^a]\cup[\Lambda_-^o,\Lambda_+^o]\subset[0,10)\cup(12,60)$ (with uniform inclusion as $h\to 0$ due to the fact that
$\Lambda_+^a$ and $\Lambda_{\pm}^o$ do not depend on $h$),  and that for $\Lambda\in[0,10)\cup(12,60)$, the weight $W$ is strictly positive, we can guarantee that
$$C_-(\eta_+^a,\eta_-^o,\eta_+^o):=\min\limits_{\Lambda\in[0,\Lambda_+^a)\cup(\Lambda_-^o,\Lambda_+^o)}W(\Lambda)\mbox{ is strictly positive.}$$
Then the proof concludes by taking in (\ref{p2ObservabilityInequalityGeneral}) with $B_h$ given by (\ref{ObservationOperators}) the observability constant $C_h(T)=C_-(T)C_-(\eta_+^a,\eta_-^o,\eta_+^o)$ which is independent of $h$.
\end{proof}

\textbf{The admissibility inequality.} Using the \textit{direct inequality} in Ingham Theorem  (cf. \cite{LorKom}, pp. 60, (4.8)), we can also prove that the inequality (\ref{p2DirectInequalityGeneral}), with $B_h$ as in (\ref{ObservationOperators}), holds uniformly as $h\to 0$ for all $T>0$, with $1/c_h(T)=C_+(T)C_+(\eta_+^a,\eta_-^o,\eta_+^o)$, where $C_+(T)$ is the constant of the direct -Ingham inequality and
$$C_+(\eta_+^a,\eta_-^o,\eta_+^o):=\max\limits_{\Lambda\in[0,\Lambda_+^a)\cup(\Lambda_-^o,\Lambda_+^o)}W(\Lambda)\mbox{ is bounded.}$$

\section{Discrete observability inequality: a bi-grid algorithm} \label{SectBigrid}

\textbf{1. The observability inequality.} In this section, $N$ will be an \textit{odd} number. We consider the space $\mathcal{L}_h$ containing \textit{piecewise linear functions} given below 
$$\mathcal{L}_h:=\{\mathbf{F}^h=(F_{j/2})_{1\leq j\leq 2N+1}\mbox{ with }F_0=F_{N+1}=0,\mbox{ s.t. }F_{j+1/2}=(F_j+F_{j+1})/2,\
\forall 0\leq j\leq N\}$$
and the space $\mathcal{B}_h$ of the discrete functions whose \textit{nodal components} are given by a \textit{bi-grid algorithm}, i.e. the even components are arbitrarily given and the odd ones are computed as average of the two even neighboring values:
$$\mathcal{B}_h:=\{\mathbf{F}^h=(F_{j/2})_{1\leq j\leq 2N+1}\mbox{ with }F_0=F_{N+1}=0,\mbox{ s.t. }F_{2j+1}=(F_{2j}+F_{2j+2})/2,\ \forall 0\leq j\leq(N-1)/2\}.$$
We also define the subspace $\mathcal{S}_h$ of $\mathcal{V}_h$
\begin{equation}\mathcal{S}_h:=((\mathcal{L}_h\cap\mathcal{B}_h)\times(\mathcal{L}_h\cap\mathcal{B}_h))\cap\mathcal{V}_h.\label{p2SubsBigrid}\end{equation}

The aim of this section is to prove that the observability inequality (\ref{p2ObservabilityInequalityGeneral}) still holds uniformly as
$h\to 0$ for initial data in the bi-grid subspace $\mathcal{S}_h$ introduced in (\ref{p2SubsBigrid}):
\begin{theorem}\label{TheoremBigrid}For all $T> 2$ and all initial data $(\mathbf{U}_h^0,\mathbf{U}_h^1)$ in the adjoint problem (\ref{p2adjoint}) belonging to $\mathcal{S}_h$ introduced in (\ref{p2SubsBigrid}), the observability inequality (\ref{p2ObservabilityInequalityGeneral}) with $B_h$ given by  (\ref{ObservationOperators})
holds uniformly as $h\to 0$. \end{theorem}

\begin{remark} Note that, in the bi-grid filtering mechanism we have designed, the data under consideration have been taken, before filtering through the classical bi-grid algorithm, to be piecewise linear in each interval $(x_j,x_{j+1})$, $j\in\zz$, which imposes a further restriction. This allows to obtain the sharp observability time. 

The bi-grid filtering algorithm proposed in Theorem \ref{TheoremBigrid} yields optimal observability time, i.e. the characteristic one $T^{\star}=2$. This is due to the fact that for a numerical scheme the minimal time required for the observability to hold is $2/v$,
where $v$ is the minimal group velocity involved in the corresponding solution. From our analysis, we will see that the bi-grid filtering algorithm above acts mainly as a Fourier truncation of the whole optic diagram and of the second half (the high frequency one) of the acoustic one.
Consequently, $v:=\min_{\eta\in[0,\pi/2]}\partial_{\eta}\lambda^{a}(\eta)$. Since the group velocity of the acoustic branch,
$\partial_{\eta}\lambda^a(\eta)$, is increasing on $[0,\pi/2]$, we conclude that $v=\partial_{\eta}\lambda^a(0)=1$ and then the observability time of the numerical scheme is sharp: $T>2$. 
\end{remark}
The following two auxiliary results hold:
\begin{proposition}If the initial data $(\mathbf{U}_h^0,\mathbf{U}_h^1)$ in (\ref{p2adjoint})
belong to $(\mathcal{L}_h\times\mathcal{L}_h)\cap\mathcal{V}_h$, then the resonant Fourier coefficients in (\ref{p2adjointSol}) vanish, i.e.
\begin{equation}\widehat{u}^r_{\pm}=0\label{nullresonant}\end{equation}
and the optic and acoustic ones are related by the following two identities:
\begin{equation}(\widehat{u}^{a,k}_++\widehat{u}^{a,k}_-)\big(m^{a,k}-n^{a,k}\cos\big(\frac{k\pi h}{2}\big)\big)
+(\widehat{u}^{o,k}_++\widehat{u}^{o,k}_-)\big(m^{o,k}-n^{o,k}\cos\big(\frac{k\pi h}{2}\big)\big)=0\label{opticalacoustical1}\end{equation}
and
\begin{equation}\lambda^{a,k}(\widehat{u}^{a,k}_+-\widehat{u}^{a,k}_-)\big(m^{a,k}-n^{a,k}\cos\big(\frac{k\pi h}{2}\big)\big)
+\lambda^{o,k}(\widehat{u}^{o,k}_+-\widehat{u}^{o,k}_-)\big(m^{o,k}-n^{o,k}\cos\big(\frac{k\pi h}{2}\big)\big)=0.\label{opticalacoustical2}\end{equation}
\label{PropLinearData}\end{proposition}

Taking squares in (\ref{opticalacoustical1}) and (\ref{opticalacoustical2}) and in view of 
(\ref{p2eignormalized}), we deduce that
\begin{equation}|\widehat{u}^{o,k}_++\widehat{u}^{o,k}_-|^2=\frac{W_1(\Lambda^{a,k})}{W_1(\Lambda^{o,k})}
|\widehat{u}^{a,k}_++\widehat{u}^{a,k}_-|^2\mbox{ and }|\widehat{u}^{o,k}_+-\widehat{u}^{o,k}_-|^2=\frac{\Lambda^{a,k}}{\Lambda^{o,k}}
\frac{W_1(\Lambda^{a,k})}{W_1(\Lambda^{o,k})}
|\widehat{u}^{a,k}_+-\widehat{u}^{a,k}_-|^2,\label{opticalacoustical}\end{equation}
where $$W_1(\Lambda)=\frac{\Lambda^2(\Lambda^2+16\Lambda+240)}{(\Lambda-10)(19\Lambda^2+120\Lambda-3600)}.$$

\begin{proof}[Proof of Proposition \ref{PropLinearData}] We will prove only (\ref{nullresonant}) and (\ref{opticalacoustical1}), the proof of
(\ref{opticalacoustical2}) being similar to the one for (\ref{opticalacoustical1}). Observe that the Fourier representation of the identity that characterizes
$\mathbf{U}_h^0\in\mathcal{L}_h$ is
\begin{align}\label{bigrid1}0=&\sum\limits_{k=1}^N[(\widehat{u}^{a,k}_++\widehat{u}^{a,k}_-)(m^{a,k}-n^{a,k}\cos(k\pi h/2))+
(\widehat{u}^{o,k}_++\widehat{u}^{o,k}_-)(m^{o,k}-n^{o,k}\cos(k\pi h/2))]\sin(k\pi x_{j+1/2})\\&+(\widehat{u}^r_++\widehat{u}^r_-)
\frac{\sqrt{15}}{2\sqrt{2}}\sin((N+1)\pi x_{j+1/2}),\nonumber\end{align}
for all $0\leq j\leq N$. Multiplying (\ref{bigrid1}) by $\sin(l\pi x_{j+1/2})$, $1\leq l\leq N+1$, adding in $0\leq j\leq N$ and taking into account that
$h\sum_{j=0}^N\sin(k\pi x_{j+1/2})\sin(l\pi x_{j+1/2})=\delta_{k,l}/2$, for all $1\leq k,l\leq N+1$, we conclude the two identities (\ref{nullresonant})
and (\ref{opticalacoustical1}).
\end{proof}
The total energy of data $(\mathbf{U}_h^0,\mathbf{U}_h^1)\in(\mathcal{L}_h\times\mathcal{L}_h)\cap\mathcal{V}_h$ in (\ref{p2adjoint}) can be written only in terms of the nodal components and coincides with the one of the $P_1$-finite element method
\begin{equation}\mathcal{E}_h(\mathbf{U}_h^0,\mathbf{U}_h^1)=\frac{h}{2}\sum\limits_{j=0}^N\Big|\frac{U^0_{j+1}-U^0_j}{h}\Big|^2+\frac{h}{12}\sum\limits_{j=0}^N
(2|U_j^1|^2+|U_{j+1}^1+U_j^1|^2).\label{p2energyLinearData}\end{equation}
Taking into account the form of the Fourier coefficients (\ref{nullresonant}) and (\ref{opticalacoustical}) corresponding to linear initial data, we obtain that the Fourier representation of the total energy (\ref{p2energyLinearData}) is
as follows:
\begin{equation}\mathcal{E}_h(\mathbf{U}_h^0,\mathbf{U}_h^1)=\frac{1}{2}
\sum\limits_{k=1}^N\Lambda^{a,k}_h\Big[\Big(1+\frac{\Lambda^{o,k}}{\Lambda^{a,k}}\frac{W_1(\Lambda^{a,k})}{W_1(\Lambda^{o,k})}\Big)
|\widehat{u}^{a,k}_++\widehat{u}^{a,k}_-|^2
+\Big(1+\frac{W_1(\Lambda^{a,k})}{W_1(\Lambda^{o,k})}\Big)
|\widehat{u}^{a,k}_+-\widehat{u}^{a,k}_-|^2\Big].\label{p2energyLinearDataFourier}\end{equation}

The second auxiliary result establishes that for initial data in the bi-grid subspace $\mathcal{S}_h$ in (\ref{p2SubsBigrid}), the high frequency Fourier coefficients on the acoustic branch can be evaluated in terms of the low frequency ones:
\begin{proposition}\label{PropLowHighFreqAcoustical}For each element
$(\mathbf{U}_h^0,\mathbf{U}_h^1)\in\mathcal{S}_h$ introduced in (\ref{p2SubsBigrid}), the following identities hold:
\begin{equation}\widehat{u}^{a,(N+1)/2}_+=\widehat{u}^{a,(N+1)/2}_-=0,\label{nullmidfrequency}\end{equation}
\begin{equation}\widehat{u}^{a,N+1-k}_++\widehat{u}^{a,N+1-k}_-=-\frac{n^{a,N+1-k}}{n^{a,k}}
\frac{W_2(\Lambda^{a,N+1-k})}{W_2(\Lambda^{a,k})}
(\widehat{u}^{a,k}_++\widehat{u}^{a,k}_-),\label{lowhighfreqacou1}\end{equation}
and
\begin{equation}\lambda^{a,N+1-k}(\widehat{u}^{a,N+1-k}_+-\widehat{u}^{a,N+1-k}_-)=-\frac{n^{a,N+1-k}}{n^{a,k}}
\frac{W_2(\Lambda^{a,N+1-k})}{W_2(\Lambda^{a,k})}
\lambda^{a,k}(\widehat{u}^{a,k}_+-\widehat{u}^{a,k}_-),\label{lowhighfreqacou2}\end{equation}
for all $1\leq k\leq (N-1)/2$, with
$$W_2(\Lambda)=\frac{(60-\Lambda)(\Lambda-10)(\Lambda-12)}{(\Lambda^2+16\Lambda+240)^2}.$$\end{proposition}
Taking squares in (\ref{lowhighfreqacou1}-\ref{lowhighfreqacou2}), we obtain that
\begin{equation}|\widehat{u}^{a,N+1-k}_++\widehat{u}^{a,N+1-k}_-|^2=\frac{W_3(\Lambda^{a,N+1-k})}{W_3(\Lambda^{a,k})}
|\widehat{u}^{a,k}_++\widehat{u}^{a,k}_-|^2\label{lowhighfreqacou3}\end{equation}
and
\begin{equation}|\widehat{u}^{a,N+1-k}_+-\widehat{u}^{a,N+1-k}_-|^2=\frac{\Lambda^{a,k}}{\Lambda^{a,N+1-k}}
\frac{W_3(\Lambda^{a,N+1-k})}{W_3(\Lambda^{a,k})}|\widehat{u}^{a,k}_+-\widehat{u}^{a,k}_-|^2,\ \forall 1\leq k\leq (N-1)/2,\label{lowhighfreqacou4}\end{equation}
where $$W_3(\Lambda)=W^2_2(\Lambda)\frac{(\Lambda-10)(\Lambda^2+16\Lambda+240)}{19\Lambda^2+120\Lambda-3600}=\frac{(60-\Lambda)^2(\Lambda-10)^3(\Lambda-12)^2}
{(19\Lambda^2+120\Lambda-3600)(\Lambda^2+16\Lambda+240)^3}.$$
\begin{proof}[Proof of Proposition \ref{PropLowHighFreqAcoustical}] Due to the orthogonality in $\mathcal{H}_h^0$ of the eigenvectors in $\mathcal{EF}_h$, to the fact that $\mathbf{U}_h^0\in\mathcal{L}_h$ and using the representation (\ref{p2eignormalized}) of the normalized eigenvectors, the following identity holds
\begin{equation}\widehat{u}^{a,k}_++\widehat{u}^{a,k}_-=(\mathbf{U}_h^0,\mathbf{\varphi}^{a,k}_h)_{h,0}=
\frac{4(60-\Lambda^{a,k})}{(\Lambda^{a,k})^2+16\Lambda^{a,k}+260}n^{a,k}h\sum\limits_{j=1}^NU^0_j\sin(k\pi x_j).\label{proof1}\end{equation}

Now, using the fact that $\mathbf{U}_h^0\in\mathcal{B}_h$, we have
$$h\sum\limits_{j=1}^NU^0_j\sin(k\pi x_j)=2\cos^2\big(\frac{k\pi h}{2}\big)h\sum\limits_{j=1}^{(N-1)/2}U_{2j}^0\sin(k\pi x_{2j}).$$
Taking  (\ref{p2spectral5}) into account, we have
\begin{equation}\widehat{u}^{a,k}_++\widehat{u}^{a,k}_-=16W_2(\Lambda^{a,k})n^{a,k}h\sum\limits_{j=1}^{(N-1)/2}U_{2j}^0\sin(k\pi x_{2j}).\label{proof2}
\end{equation}
For $k=(N+1)/2$, we obtain $\widehat{u}^{a,(N+1)/2}_++\widehat{u}^{a,(N+1)/2}_-=0$.  Since
$\sin((N+1-k)\pi x_{2j})=-\sin(k\pi x_{2j})$, we obtain (\ref{lowhighfreqacou1}) by equating the expressions of
$h\sum_{j=1}^{(N-1)/2}U_{2j}^0\sin(k\pi x_{2j})$
form the identity (\ref{proof2}) corresponding to $k$ and to $N+1-k$. The proof of (\ref{lowhighfreqacou2}) is similar, based on the
fact that $(\mathbf{U}_h^1,\mathbf{\varphi}^{a,k})_{h,0}=i\lambda^{a,k}_h(\widehat{u}^{a,k}_+-\widehat{u}^{a,k}_-)$, from which for $k=(N+1)/2$
we obtain that $\widehat{u}^{a,(N+1)/2}_+-\widehat{u}^{a,(N+1)/2}_-=0$, which concludes (\ref{nullmidfrequency}).
\end{proof}
Replacing the Fourier representations (\ref{lowhighfreqacou3}) and (\ref{lowhighfreqacou4}) into the total energy with linear initial data (\ref{p2energyLinearDataFourier}),
we obtain that energy corresponding to initial data $(\mathbf{U}_h^0,\mathbf{U}_h^1)\in\mathcal{S}_h$ in (\ref{p2SubsBigrid}) is given by
\begin{equation}\mathcal{E}_h(\mathbf{U}_h^0,\mathbf{U}_h^1)=\frac{1}{2}
\sum\limits_{k=1}^{(N-1)/2}\Lambda_h^{a,k}[(W_{+,k}^{lo}+W_{+,k}^{hi})|\widehat{u}^{a,k}_++\widehat{u}^{a,k}_-|^2
[(W_{-,k}^{lo}+W_{-,k}^{hi})|\widehat{u}^{a,k}_+-\widehat{u}^{a,k}_-|^2],\label{p2energLinearBigridDataFourier}\end{equation}
where the low and high frequency coefficients are given by
$$W_{+,k}^{lo}:=1+\frac{\Lambda^{o,k}}{\Lambda^{a,k}}\frac{W_1(\Lambda^{a,k})}{W_1(\Lambda^{o,k})},\quad W_{-,k}^{lo}:=1+\frac{W_1(\Lambda^{a,k})}{W_1(\Lambda^{o,k})},
\quad W_{-,k}^{hi}:=\Big(1+\frac{W_1(\Lambda^{a,N+1-k})}{W_1(\Lambda^{o,N+1-k})}\Big)\frac{W_3(\Lambda^{a,N+1-k})}{W_3(\Lambda^{a,k})}$$
and
$$W_{+,k}^{hi}:=\frac{\Lambda^{a,N+1-k}}{\Lambda^{a,k}}\Big(1+\frac{\Lambda^{o,N+1-k}}{\Lambda^{a,N+1-k}}
\frac{W_1(\Lambda^{a,N+1-k})}{W_1(\Lambda^{o,N+1-k})}\Big)\frac{W_3(\Lambda^{a,N+1-k})}{W_3(\Lambda^{a,k})}.$$

For any $\delta\in(0,1)$ which does not depend on $h$ and any solution (\ref{p2adjointSol}) of (\ref{p2adjoint}), let us introduce its projection
on the first $\delta(N+1)$ frequencies of the acoustic branch to be
\begin{equation}\Gamma^{a}_{\delta}\mathbf{U}_h(t)=\sum\limits_{\pm}\sum\limits_{k=1}^{\delta(N+1)}\widehat{u}^{a,k}_{\pm}
\exp(\pm it\lambda^{a,k}_h)\mathbf{\varphi}^{a,k}_h.\label{p2projacoudelta}\end{equation}
Observe that the projection (\ref{p2projacoudelta}) is still a solution of (\ref{p2adjoint}), therefore its total energy is conserved in time. Set
 $$\mathcal{E}_h(\Gamma^{a}_{\delta}(\mathbf{U}_h^0,\mathbf{U}_h^1)):=
\mathcal{E}_h(\Gamma^{a}_{\delta}\mathbf{U}_h(0),\Gamma^{a}_{\delta}\mathbf{U}_{h,t}(0))=\sum\limits_{k=1}^{\delta(N+1)}
\Lambda_h^{a,k}(|\widehat{u}^{a,k}_+|^2+|\widehat{u}^{a,k}_-|^2).$$
The following result provides a bound of the total energy of the solutions of (\ref{p2adjoint}) with initial data $(\mathbf{U}_h^0,\mathbf{U}_h^1)\in\mathcal{S}_h$ as in (\ref{p2SubsBigrid}) in terms of the total energy of their projections on the first half of the acoustic mode:
\begin{proposition}For any solution $\mathbf{U}^h(t)$ of (\ref{p2adjoint}) with initial data $(\mathbf{U}_h^0,\mathbf{U}_h^1)\in\mathcal{S}_h$ introduced in (\ref{p2SubsBigrid}), there exists a constant $C>0$ which does not depend on $h$ such that
\begin{equation}\mathcal{E}_h(\mathbf{U}_h^0,\mathbf{U}_h^1)\leq C
\mathcal{E}_h(\Gamma^a_{1/2}(\mathbf{U}_h^0,\mathbf{U}_h^1)).\label{EstimProjAcousticalLowFreq}
\end{equation}\label{PropProjAcousticalLowFreq}\end{proposition}
\begin{proof}[Proof of Proposition \ref{PropProjAcousticalLowFreq}] In order to obtain the estimate (\ref{EstimProjAcousticalLowFreq}),
we claim that it is sufficient
to prove that there exist $W_{\pm}^{lo}, W_{\pm}^{hi}>0$ independent of $h$ such that $W_{\pm,k}^{lo}\leq W_{\pm}^{lo}$ and
 $W_{\pm,k}^{hi}\leq W_{\pm}^{hi}$, for all $1\leq k\leq (N-1)/2$. Assuming this for a moment, we can take
$C:=\max\{W_{\pm}^{lo},W_{\pm}^{hi}\}$ for which (\ref{EstimProjAcousticalLowFreq}) holds.

Let us analyze the boundedness of the four coefficients. Observe that $W_{-,k}^{lo}$ involves the product of $W_1(\Lambda^{a,k})$ and $1/W_1(\Lambda^{o,k})$ for $kh\leq 1/2$.
But $W_1(\Lambda)$ is singular only for $\Lambda\to 10$, whereas for $kh\leq 1/2$, due to the increasing nature of $\Lambda^a$,
we have $\Lambda^{a,k}\leq \Lambda^a(\pi/2)=60/(13+2\sqrt{31})<3$. Also $1/W_1(\Lambda)$ is singular as $\Lambda\to 0$, but since $\Lambda^o$ is decreasing in $\eta$
we have $\Lambda^{o,k}\geq \Lambda^o(\pi/2)=(52+8\sqrt{31})/3>30$ for all $kh\leq 1/2$. In the same way, $W_{+,k}^{lo}$ contains the product of
$W_1(\Lambda^{a,k})/\Lambda^{a,k}$ with $\Lambda^{o,k}/W_1(\Lambda^{o,k})$, for $kh\leq 1/2$. The second factor $\Lambda/W_1(\Lambda)$ has a singularity as
$\Lambda\to 0$, but is evaluated for $\Lambda=\Lambda^{o,k}$ which is far from the singularity for all $kh\leq 1/2$. Since $W_1(\Lambda)$ contains a factor
$\Lambda^2$, we deduce that $W_1(\Lambda)/\Lambda$ has the same singularities as $W_1(\Lambda)$, i.e. $\Lambda=10$, but since we work on the first half
of the acoustic diagram, $\Lambda^{a,k}$ is far from that singularity. We conclude the existence
of $W_{\pm}^{lo}>0$ independent of $h$ such that $W_{\pm,k}^{lo}\leq W_{\pm}^{lo}$.

\smallskip The coefficient $W_{-,k}^{hi}$ contains two terms. The first of them is constituted by the factors $W_3(\Lambda^{a,N+1-k})$ and $1/W_3(\Lambda^{a,k})$.
Since $W_3(\Lambda)$ is not singular for any $\Lambda\in(0,10)\cup(12,60)$, then $W_3(\Lambda^{a,N+1-k})$ is bounded for all $kh\leq 1/2$.
On the other hand, $1/W_3(\Lambda)$ has three singularities: $\Lambda=10$, $\Lambda=12$ and $\Lambda=60$. Since all $\Lambda$-s situated on the first half of the acoustic mode are
far for all the three singularities, we deduce the boundedness of $1/W_3(\Lambda^{a,k})$. The second term is a product of four factors:
$W_1(\Lambda^{a,N+1-k})$, $W_3(\Lambda^{a,N+1-k})$, $1/W_1(\Lambda^{o,N+1-k})$ and $1/W_3(\Lambda^{a,k})$. In view of our previous analysis,
we deduce the boundedness of the last three factors. The first factor blows-up like $1/(10-\Lambda^{a,N+1-k})$ for small $k$. Nevertheless, in the same range of
$k$-s, $W_3(\Lambda^{a,N+1-k})$ behaves like $(10-\Lambda^{a,N+1-k})^3$, compensating in this way the singularity of the first factor $W_1(\Lambda^{a,N+1-k})$,
so that $W_1(\Lambda^{a,N+1-k})W_3(\Lambda^{a,N+1-k})$ is bounded for all $kh\leq 1/2$.

\smallskip The coefficient $W_{+,k}^{hi}$ also contains two terms. The first of them includes the factors: $\Lambda^{a,N+1-k}$, $1/\Lambda^{a,k}$, $W_3(\Lambda^{a,N+1-k})$
and $1/W_3(\Lambda^{a,k})$. We have already analyzed the first, the third and the fourth ones. The second one blows-up like
$\sin^{-2}(k\pi h/2)$ for small $k$. But, as we said, $W_3(\Lambda^{a,N+1-k})$ behaves like $(10-\Lambda^{a,N+1-k})^3\sim \sin^6(k\pi h/2)$ for small
$k$, compensating the singularity of the second factor, so that $W_3(\Lambda^{a,N+1-k})/\Lambda^{a,k}$ is bounded for all $kh\leq 1/2$. The second term
in $W_{+,k}^{hi}$ contains the factors $1/\Lambda^{a,k}$, $\Lambda^{o,N+1-k}$, $W_3(\Lambda^{a,N+1-k})$ and $1/W_3(\Lambda^{a,k})$ and is bounded by the same arguments
we used for the first term. Consequently, there exist $W_{\pm}^{hi}>0$ such that $W_{\pm,k}^{hi}\leq W_{\pm}^{hi}$, which concludes the proof.
\end{proof}
\begin{remark}\label{remark3}Set $\mathcal{L}_h^{\alpha}:
=\{\mathbf{F}^h=(F_{j/2})_{1\leq j\leq 2N+1}\mbox{ s.t. }F_{j+1/2}=\alpha(F_j+F_{j+1}),\ 0\leq j\leq N\}$ and 
$\mathcal{S}_h^{\alpha}:=
((\mathcal{L}_h^{\alpha}\cap\mathcal{B}_h)\times(\mathcal{L}_h^{\alpha}\cap\mathcal{B}_h))\cap\mathcal{V}_h$. Observe that
$\mathcal{L}_h=\mathcal{L}_h^{1/2}$ and $\mathcal{S}_h=\mathcal{S}_h^{1/2}$, where $\mathcal{S}_h$ is the space in (\ref{p2SubsBigrid}). We want to point out that the result of Theorem \ref{TheoremBigrid}
is not longer true when replace $\mathcal{S}_h$ by $\mathcal{S}_h^{\alpha}$, with $\alpha\not=1/2$, so that the condition on the initial data to be linear
is sharp. Indeed, when replacing $\mathcal{S}_h$ by $\mathcal{S}_h^{\alpha}$, in particular
$W_{+,k}^{lo}$ in (\ref{p2energLinearBigridDataFourier}) has to be substituted by
$$W_{+,k}^{lo,\alpha}:=1+\frac{\Lambda^{o,k}}{\Lambda^{a,k}}\frac{W_1^{\alpha}(\Lambda^{a,k})}{W_1^{\alpha}(\Lambda^{o,k})},\mbox{ with }
W_1^{\alpha}(\Lambda)=\frac{1}{25}\frac{(\Lambda^2+16\Lambda+240)(40-80\alpha+(1+8\alpha)\Lambda)^2}{(\Lambda-10)(19\Lambda^2+120\Lambda-3600)}.$$
One can show that, for $\alpha\not=1/2$, it is not longer true that $W_1^{\alpha}(\Lambda^{a,k})\to 0$ as $kh\to 0$, so that this factor cannot compensate
the singularity of $1/\Lambda^{a,k}$ as $kh\to 0$, like for $\alpha=1/2$. Consequently, for $\alpha\not=1/2$, at least $W_{+,k}^{lo,\alpha}$ is not bounded for
$1\leq k\leq (N-1)/2$.
\end{remark}

\begin{proof}[Proof of Theorem \ref{TheoremBigrid}] Proposition \ref{PropLowHighFreqAcoustical} ensures that the total energy of initial data in
$\mathcal{S}_h$ introduced in (\ref{p2SubsBigrid}) is uniformly bounded by the energy of their projection on the first half of the acoustic mode. On the other hand,  
Theorem \ref{TheoremFourierTruncation} guarantees that the observability inequality (\ref{p2ObservabilityInequalityGeneral}) with $B_h$ as in (\ref{ObservationOperators})
holds uniformly as $h\to 0$ in the class of truncated data lying on the first half of the acoustic mode.  Combining these two facts, one can apply a dyadic decomposition argument as in \cite{IgZuaWave} and conclude the proof of Theorem \ref{TheoremBigrid}.
\end{proof}

\textbf{2. The admissibility inequality.} In the rest of this section, our aim is to prove the direct inequality (\ref{p2DirectInequalityGeneral})
for the boundary operator $B_h$ in (\ref{ObservationOperators}) using the spectral identities (\ref{ObservabilityEigenvectorsao})
and (\ref{ObservabilityEigenvectorsr}). Firstly, let us observe that for all matrix operator $B_h$ and all $T>0$, we have the following inequality:
\begin{align}\label{directineq1}\int\limits_0^T||B_h\mathbf{U}_h(t)||^2_{\cc^{2N+1}}\,dt&
=\int\limits_0^T((B_hS_h^{-1/2})^*B_hS_h^{-1/2}S_h^{1/2}\mathbf{U}_h(t),S_h^{1/2}\mathbf{U}_h(t))_{\cc^{2N+1}}\,dt
\\&\leq ||(B_hS_h^{-1/2})^*B_hS_h^{-1/2}||_{\cc^{2N+1}}\int\limits_0^T(S_h\mathbf{U}_h(t),\mathbf{U}_h(t))_{\cc^{2N+1}}\,dt\nonumber\\&\leq 2T\mathcal{E}_h(\mathbf{U}_h^0,\mathbf{U}_h^1)||(B_hS_h^{-1/2})^*B_hS_h^{-1/2}||_{\cc^{2N+1}}.\nonumber\end{align}
For any matrix $A$, the matrix norm $||\cdot||_{\cc^{2N+1}}$ involved in the right hand side of (\ref{directineq1}) is defined as
$$||A^*A||_{\cc^{2N+1}}=\max_{||\mathbf{\varphi}_h||_{h,0}=1}||AM_h^{1/2}\mathbf{\varphi}_h||_{\cc^{2N+1}}.$$
In the above definition of the norm of a matrix, we can reduce the set of test functions to
$\mathbf{\varphi}_h\in\mathcal{EF}_h$ introduced in (\ref{p2SetEigenvectors}), which is an orthonormal basis in $\rr^{2N+1}$.
Let us remark that for any eigenfunction $\mathbf{\varphi}_h\in\mathcal{EF}_h$, the corresponding eigenvalue $\Lambda_h\in\mathcal{EV}_h$
verifies the identity $S_h^{-1/2}M_h^{1/2}\mathbf{\varphi}_h=\Lambda_h^{-1/2}\mathbf{\varphi}_h.$
Consequently, for any matrix $B_h$ and $\mathbf{\varphi}_h\in\mathcal{EF}_h$, we have
\begin{equation}||B_hS_h^{-1/2}M_h^{1/2}\mathbf{\varphi}_h||_{\cc^{2N+1}}=\Lambda_h^{-1/2}||B_h\mathbf{\varphi}_h||_{\cc^{2N+1}}.\label{directineq2}
\end{equation}

For $B_h$ in (\ref{ObservationOperators}), using the identities (\ref{directineq2}), (\ref{ObservabilityEigenvectorsao})
and (\ref{ObservabilityEigenvectorsr}), we conclude that
$$||(B_hS_h^{-1/2})^*B_hS_h^{-1/2}||_{\cc^{2N+1}}^2=\max_{\mathbf{\varphi}_h\in\mathcal{EF}_h}\frac{1}{\Lambda_h}\Big|\frac{\varphi_N}{h}\Big|^2
=\max\Big\{\max_{\Lambda\in(0,10)\cup(12,60)}W(\Lambda),\frac{3}{16}\Big\}$$
is a quantity independent of $h$.

\section{Convergence of the discrete controls}\label{SectConvergence}

In this section, we describe the algorithm of constructing the discrete controls of minimal $L^2(0,T)$-norm and we prove their convergence
towards the continuous HUM boundary control $\tilde{v}(t)$ in (\ref{contHUMcontrol}) as $h\to 0$, under the hypothesis that both
inverse and direct inequalities (\ref{p2ObservabilityInequalityGeneral}) and (\ref{p2DirectInequalityGeneral}) hold uniformly as $h\to 0$. As we saw in the previous sections, the above hypothesis holds when the initial data in the adjoint system (\ref{p2adjoint}) 
is filtered through a Fourier truncation or a bi-grid algorithm. 

\subsection{Description of the algorithm.} Using the \textit{admissibility inequality} (\ref{p2DirectInequalityGeneral})
and the \textit{observability one} (\ref{p2ObservabilityInequalityGeneral}), one can prove the \textit{continuity} and the \textit{uniform coercivity}
 of $\mathcal{J}_h$  defined by  (\ref{p2FunctionalGeneral}) on $\mathcal{S}_h\subset\mathcal{V}_h$, where $\mathcal{S}_h$ can be both the truncated space (\ref{p2SubspTruncation}) or the bi-grid one (\ref{p2SubsBigrid}). Moreover, it is \textit{strictly convex}.
Applying the \textit{direct method for the calculus of variations} (DMCV) (cf. \cite{Dacorogna}), one can guarantee the existence of an unique
minimizer
$(\mathbf{\tilde{U}}_h^0,\mathbf{\tilde{U}}_h^1)$ of $\mathcal{J}_h$, i.e.:
\begin{equation}\mathcal{I}_h:=\mathcal{J}_h(\mathbf{\tilde{U}}_h^0,\mathbf{\tilde{U}}_h^1)=
\min\limits_{(\mathbf{U}_h^0,\mathbf{U}_h^1)\in\mathcal{S}_h}\mathcal{J}_h(\mathbf{U}_h^0,\mathbf{U}_h^1).\label{minimizerJh}\end{equation}

Moreover, the Euler-Lagrange equation (\ref{p2EulerLagrange}) associated to $\mathcal{J}_h$
 characterizes the optimal control (\ref{p2OptimalControl}).

 Remark that when the space of initial data in (\ref{p2adjoint})  is restricted to a subspace $\mathcal{S}_h\subset\mathcal{V}_h$, for example the ones given by (\ref{p2SubspTruncation}) or (\ref{p2SubsBigrid}), the exact controllability condition (\ref{p2ExactControl}) holds for all
$(\mathbf{U}_h^0,\mathbf{U}_h^1)\in\mathcal{S}_h$. This does not imply that the final state
$(\mathbf{Y}_{h,t}(T),-\mathbf{Y}_h(T))$ in the controlled problem (\ref{p2controlledpbm}) is exactly controllable to the rest, but its orthogonal projection $\Gamma_{\mathcal{S}_h}$
from $\mathcal{V}_h$ on the subspace $\mathcal{S}_h$, i.e. 
$$\Gamma_{\mathcal{S}_h}(\mathbf{Y}_{h,t}(T),-\mathbf{Y}_h(T))=0\mbox{, meaning that }\langle(\mathbf{Y}_{h,t}(T),-\mathbf{Y}_h(T)),(\mathbf{U}^0_h,\mathbf{U}^1_h)\rangle_{\mathcal{V}_h',\mathcal{V}_h}=0,\ \forall (\mathbf{U}^0_h,\mathbf{U}^1_h)\in\mathcal{S}_h.$$

\subsection{The convergence result.} Set $\tilde{v}_h(t)$ to be the last component of the control $\mathbf{\tilde{V}}_h(t)$ in (\ref{p2OptimalControl}) (the other ones vanish). Since $\mathcal{I}_h\leq\mathcal{J}_h(0,0)=0$ and taking into account the inverse inequality
(\ref{p2ObservabilityInequalityGeneral}), we obtain: 
\begin{equation}\int\limits_0^T|\tilde{v}_h(t)|^2\,dt\leq 8C(T)||(\mathbf{Y}_h^1,
-\mathbf{Y}_h^0)||_{\mathcal{V}_h'}^2,\label{boundp2control}\end{equation}
where $C(T)=\sup\limits_{h\in(0,1)}C_h(T)$ and $C_h(T)$ is the observability constant in (\ref{p2ObservabilityInequalityGeneral}) under filtering.

Set $\mathcal{EF}:=\{\varphi^k(x)=\sqrt{2}\sin(k\pi x)\}$, $\Lambda^k:=k^2\pi^2$,
$\lambda^k:=k\pi$ and $\ell^2$ be the space of square summable sequences. Since $\mathcal{EF}$ is a Hilbertian basis in each $H^s(0,1)$, $s\in\rr$, the initial data $(y^1,-y^0)\in\mathcal{V}$ to be controlled in the continuous problem
(\ref{contwavecontrolled}) admits the following Fourier decomposition:
\begin{equation}y^i(x)=\sum\limits_{k=1}^{\infty}\widehat{y}^{k,i}\varphi^k(x),\mbox{ with }\widehat{y}^{k,i}=(y^i,\varphi^k)_{L^2},\ i=0,1.\label{contFourierDecompositionInitialDataControl}\end{equation}
Moreover, their $||\cdot||_{\mathcal{V}'}$-norm has the following Fourier representation:
\begin{equation}||(y^1,-y^0)||_{\mathcal{V}'}^2=\sum\limits_{k=1}^{\infty}\left(\frac{|\widehat{y}^{k,1}|^2}{\Lambda^k}+|\widehat{y}^{k,0}|^2\right).\end{equation}
Since the set $\mathcal{EF}_h$ introduced in (\ref{p2SetEigenvectors}) is a basis in
$\rr^{2N+1}$, the initial data $(\mathbf{Y}_h^1,\mathbf{Y}_h^0)\in\mathcal{V}_h'$ to be controlled in the discrete problem
(\ref{p2controlledpbm}) admit the following decomposition
\begin{equation}\mathbf{Y}_h=\sum\limits_{k=1}^N\widehat{y}^{a,k,i}_h\mathbf{\varphi}^{a,k}_h
+\widehat{y}^{r,i}_h\mathbf{\varphi}^r_h+
\sum\limits_{k=1}^N\widehat{y}^{o,k,i}_h\mathbf{\varphi}^{o,k}_h,\ \forall i=0,1,\label{p2FourierDecompInitialDataControl}\end{equation}
with
$$\widehat{y}^{\alpha,k,i}_h=
(\mathbf{Y}_h^{i},\mathbf{\varphi}^{\alpha,k}_h)_{h,0},\ \alpha\in\{a,o\}, 1\leq k\leq N, \mbox{ and }
\widehat{y}^{r,i}_h=
(\mathbf{Y}_h^{i},\mathbf{\varphi}^r_h)_{h,0}, \forall i=0,1.$$
Moreover, their $||\cdot||_{\mathcal{V}_h'}$-norm can be written in terms of the Fourier coefficients as follows:
\begin{equation}||(\mathbf{Y}_h^1,-\mathbf{Y}_h^0)||_{\mathcal{V}_h'}^2=\sum\limits_{k=1}^N
\frac{|\widehat{y}_h^{a,k,1}|^2}{\Lambda_h^{a,k}}+\frac{|\widehat{y}^{r,1}_h|^2}{\Lambda^r_h}+
\sum\limits_{k=1}^N
\frac{|\widehat{y}_h^{o,k,1}|^2}{\Lambda_h^{o,k}}
+\sum\limits_{k=1}^N
|\widehat{y}_h^{a,k,0}|^2+|\widehat{y}^{r,0}_h|^2+
\sum\limits_{k=1}^N
|\widehat{y}_h^{o,k,0}|^2.\label{Vhprimnorm}\end{equation}

The main result of this section is as follows:
\begin{theorem}In the controlled problem (\ref{p2controlledpbm}), we consider initial data with the following two properties:

\begin{equation}\left(\frac{\widehat{y}^{a,k,1}_h}{\lambda^{a,k}_h}\right)_k\rightharpoonup\left(\frac{\widehat{y}^{k,1}}{\lambda^k}\right)_k\quad\mbox{and}\quad
(\widehat{y}^{a,k,0}_h)_k\rightharpoonup(\widehat{y}^{k,0})_k\quad\mbox{as}\quad h\to 0,\quad\mbox{weakly in }\ell^2, \label{p2ConvDataContrAcoustical}\end{equation}
and
\begin{equation}\left(\frac{\widehat{y}^{o,k,1}_h}{\lambda^{o,k}_h}\right)_k\rightharpoonup 0\quad\mbox{and}\quad
(\widehat{y}^{o,k,0}_h)_k\rightharpoonup 0\quad\mbox{as}\quad h\to 0,\quad\mbox{weakly in }\ell^2. \label{p2ConvDataContrOpt}\end{equation}
Then
\begin{equation}\tilde{v}_h\rightharpoonup\tilde{v}\quad\mbox{as}\quad h\to0,\quad\mbox{weakly in }L^2(0,T),\label{ConvContr}\end{equation}
where $\tilde{v}_h$ is the last component of the discrete optimal control $\mathbf{\tilde{V}}_h(t)$ given by
(\ref{p2OptimalControl}) obtained by the minimization of the functional $\mathcal{J}_h$ on the subspace 
$\mathcal{S}_h$ defined in (\ref{p2SubspTruncation}) or in (\ref{p2SubsBigrid}) and $\tilde{v}(t)$ is the continuous HUM control (\ref{contHUMcontrol}). 

Moreover, if the convergences in (\ref{p2ConvDataContrAcoustical}) and (\ref{p2ConvDataContrOpt}) are strong, then the convergence of controls in (\ref{ConvContr}) is also strong.
\label{theoremweakconvergence}
\end{theorem}

\begin{proof}[Proof of Theorem \ref{theoremweakconvergence}] Firstly, let us observe that from (\ref{p2ConvDataContrAcoustical}) and (\ref{p2ConvDataContrOpt}), we obtain that there 
exists a constant $C>0$ independent of $h$ such that 

\begin{equation}||(\mathbf{Y}_h^1,-\mathbf{Y}_h^0)||_{\mathcal{V}_h'}
\leq C.\label{p2boundeddatacontrol}\end{equation}
By combining (\ref{boundp2control}) and (\ref{p2boundeddatacontrol}), we get the uniform boundedness as $h\to 0$ of the discrete control $\tilde{v}_h$  in
$L^2(0,T)$, so that 
\begin{equation}\tilde{v}_h\rightharpoonup\tilde{v}^*\quad\mbox{as}\quad h\to 0,\quad\mbox{weakly in }L^2(0,T).\label{p2WeakConvControl}\end{equation}
It is sufficient to prove that the weak limit $\tilde{v}^*$ coincides with the continuous HUM control $\tilde{v}$ given by (\ref{contHUMcontrol}).

The control $\tilde{v}$ can be characterized as the unique control $v$  in (\ref{contwavecontrolled}) which can be expressed
as the space derivative of a solution of the adjoint problem
(\ref{contwaveadjoint}) evaluated at $x=1$. Then, we have to prove that $\tilde{v}^*$ is an admissible control of the continuous wave equation, i.e. it verifies the identity (\ref{identitycontcontrol}), and that $\tilde{v}^*=\tilde{u}^*_x(1,t)$, where $\tilde{u}^*(x,t)$ is the solution of the adjoint problem
(\ref{contwaveadjoint}) for some initial data $(\tilde{u}^{*,0},\tilde{u}^{*,1})\in\mathcal{V}$.

\textbf{Step I - The weak limit $\tilde{v}^*$ is a control in the continuous problem (\ref{contwavecontrolled}).}

\noindent Since $\{(\pm\varphi^k/i\lambda^k,\varphi^k),k\in\nn\}$ is an orthonormal basis in $\mathcal{V}:=H_0^1\times L^2(0,1)$, then the fact that
$v$ is a control in (\ref{contwavecontrolled}), so it satisfies (\ref{identitycontcontrol}), it is equivalent  to prove (\ref{identitycontcontrol})
for all initial data of the form $(u^0,u^1)=(\pm\varphi^k/i\lambda^k,\varphi^k)$, $k\in\nn$. The solution of the adjoint problem
(\ref{contwaveadjoint}) with this kind of initial data is $u(x,t)=\pm\exp(\pm i\lambda^k(t-T))\varphi^k(x)/i\lambda^k$.
The condition (\ref{identitycontcontrol}) is equivalent to the following one:
\begin{equation}\int\limits_0^Tv(t)\exp(\pm it\lambda^k)\,dt=\frac{(-1)^k}{\sqrt{2}}
\left(\frac{\widehat{y}^{k,1}}{\lambda^k}\mp i\widehat{y}^{k,0}\right),\forall k\in\nn.\label{identitycontrolFourier}\end{equation}
Let us check that $\tilde{v}^*$ satisfies (\ref{identitycontrolFourier}).  To do it, we distinguish between the two cases of subspaces $\mathcal{S}_h$ of filtered data. Firstly, let us consider the case
of truncated initial data, i.e. $\mathcal{S}_h$ is given by (\ref{p2SubspTruncation}). A particular class of initial data in $\mathcal{S}_h$
is $(\mathbf{U}_h^0,\mathbf{U}_h^1)=(\pm\mathbf{\varphi}^{a,k}_h/i\lambda^{a,k}_h,\mathbf{\varphi}^{a,k}_h),$
for which the solution of the discrete adjoint problem (\ref{p2adjoint}) is $\mathbf{U}_h(t)=
\pm\mathbf{\varphi}^{a,k}_h\exp(\pm i\lambda_h^{a,k}(t-T))/i\lambda_h^{a,k}$, for all $k\pi h\leq\eta_+^a$. From (\ref{p2EulerLagrange}),
we see that the discrete control $\tilde{v}_h$ verifies the identity
\begin{equation}\int\limits_0^T\tilde{v}_h(t)\exp(\pm it\lambda_h^{a,k})\,dt=\frac{(-1)^k}{n^{a,k}}\frac{\lambda_h^{a,k}}{\frac{\sin(k\pi h)}{h}}
\left(\frac{\widehat{y}^{a,k,1}_h}{\lambda_h^{a,k}}\mp i\widehat{y}^{a,k,0}_h\right),\forall k\pi h\leq\eta_+^a.\label{p2EulerLagrangeFourier}\end{equation}
Let us fix $k\in\nn$ (independent of $h$). In that case, $\exp(\pm it\lambda_h^{a,k})\to\exp(\pm it\lambda^k)$ as $h\to 0$ strongly in $L^2(0,T)$, so
that, taking into account the weak convergence (\ref{p2WeakConvControl}), we can pass to the
limit as $h\to 0$ in the left hand side of (\ref{p2EulerLagrangeFourier}) and we obtain the left hand side of (\ref{identitycontrolFourier}) with
$v$ substituted by $\tilde{v}^*$. On the other hand, taking into account the condition
(\ref{p2ConvDataContrAcoustical}), which is valid for all test sequences in $\ell^2$ and in particular for the basis functions of $\ell^2$,
$e^k=(0,\cdots,0,1,0,\cdots)$ (meaning that the weak convergence in $\ell^2$ is a pointwise convergence), and additionally the fact that
$n^{a,k}\to \sqrt{2}$ and $\lambda_h^{a,k}/\sin(k\pi h)/h\to 1$ as $h\to 0$ for each fixed $k$,  passing to the limit as $h\to0$ in the right hand side
of (\ref{p2EulerLagrange}), we obtain the right hand side of (\ref{identitycontrolFourier}). Then, the weak limit $\tilde{v}^*$ of the optimal control $\tilde{v}_h$
obtained by minimizing the functional $\mathcal{J}_h$ on
$\mathcal{S}_h$ in (\ref{p2SubspTruncation}) is a control for the continuous problem.

Let us consider now the case of linear initial data given by a bi-grid algorithm, i.e. $\mathcal{S}_h$ is given by (\ref{p2SubsBigrid}). Taking into account Propositions
\ref{PropLinearData} and \ref{PropLowHighFreqAcoustical}, we see that for initial data
 $(\mathbf{U}_h^0,\mathbf{U}_h^1)\in\mathcal{S}_h$ the Fourier
representation (\ref{FourierRepresentationData}) has the more particular form
\begin{equation}\mathbf{U}_h^0=\sum\limits_{k=1}^{(N-1)/2}(\widehat{u}^{a,k}_++\widehat{u}^{a,k}_-)\mathbf{\psi}^k_h
\mbox{ and }\mathbf{U}_h^1=\sum\limits_{k=1}^{(N-1)/2}(\widehat{u}^{a,k}_+-\widehat{u}^{a,k}_-)i\lambda_h^{a,k}\mathbf{\psi}^k_h.
%\mbox{ with }\widehat{u}^{a,k}_{\pm}=\frac{1}{2}\Big(\widehat{u}^{a,k}_0\pm\frac{\widehat{u}^{a,k}_1}{i\lambda_h^{a,k}}\Big).
\label{FourierRepresentationDataBigrid}\end{equation}
The basis function $\mathbf{\psi}^k_h$ for the space $\mathcal{S}_h$ is given by
\begin{equation}\mathbf{\psi}^k_h=\mathbf{\varphi}^{a,k}_h-
\frac{W_4(\Lambda^{a,k})}{W_4(\Lambda^{o,k})}\mathbf{\varphi}^{o,k}_h-\frac{W_5(\Lambda^{a,N+1-k})}{W_5(\Lambda^{a,k})}
\mathbf{\varphi}^{a,N+1-k}_h+\frac{W_5(\Lambda^{a,N+1-k})}{W_5(\Lambda^{a,k})}
\frac{W_4(\Lambda^{a,N+1-k})}{W_4(\Lambda^{o,N+1-k})}\mathbf{\varphi}^{o,N+1-k}_h,\label{BasisBigrid}\end{equation}
where
\begin{equation}W_4(\Lambda)=-\frac{\Lambda}{\Lambda-10}\sqrt{\tilde{W}(\Lambda)}\mbox{ and }
W_5(\Lambda)=\frac{(60-\Lambda)(\Lambda-10)(\Lambda-12)}{(\Lambda^2+16\Lambda+240)^2}
\sqrt{\tilde{W}(\Lambda)}\nonumber\end{equation}
and $\tilde{W}$ is defined by (\ref{normh0}). Let us fix $1\leq k\leq (N-1)/2$ and consider the homogeneous problem (\ref{p2adjoint}) with initial data
$(\mathbf{U}_h^0,\mathbf{U}_h^1)=(s\mathbf{\psi}^k/i\lambda_h^{a,k},\mathbf{\psi}^k)$, $s\in\{-1,+1\}$, for which the solution
takes the form
\begin{equation}\mathbf{U}_h(t)=\widehat{u}^{a,k}_{lo}(t)\mathbf{\varphi}^{a,k}_h
+\widehat{u}^{o,k}_{lo}(t)\mathbf{\varphi}^{o,k}_h
+\widehat{u}^{a,k}_{hi}(t)\mathbf{\varphi}^{a,N+1-k}_h
+\widehat{u}^{o,k}_{hi}(t)\mathbf{\varphi}^{o,N+1-k}_h,\label{typicalsolutionbigrid}\end{equation}
where the low frequency coefficients are $$\widehat{u}^{a,k}_{lo}(t)=s\frac{1}{i\lambda_h^{a,k}}\exp(is\lambda_h^{a,k}(t-T)),\quad
\widehat{u}^{o,k}_{lo}(t)=-\frac{W_4(\Lambda^{a,k})}{W_4(\Lambda^{o,k})}\sum\limits_{\pm}\frac{1}{2}\Big(\frac{s}{i\lambda_h^{a,k}}
\pm\frac{1}{\lambda_h^{o,k}}\Big)\exp(\pm i\lambda_h^{o,k}(t-T))$$
and the high frequency ones are
$$\widehat{u}^{a,k}_{hi}(t)=-\frac{W_5(\Lambda^{a,N+1-k})}{W_5(\Lambda^{a,k})}\sum\limits_{\pm}\frac{1}{2}\Big(\frac{s}{i\lambda_h^{a,k}}
\pm\frac{1}{\lambda_h^{a,N+1-k}}\Big)\exp(\pm i\lambda_h^{a,N+1-k}(t-T))$$
and
$$\widehat{u}^{o,k}_{hi}(t)=\frac{W_5(\Lambda^{a,N+1-k})}{W_5(\Lambda^{a,k})}
\frac{W_4(\Lambda^{a,N+1-k})}{W_4(\Lambda^{o,N+1-k})}\sum\limits_{\pm}\frac{1}{2}\Big(\frac{s}{i\lambda_h^{a,k}}
\pm\frac{1}{\lambda_h^{o,N+1-k}}\Big)\exp(\pm i\lambda_h^{o,N+1-k}(t-T)).$$

By considering the particular class of solutions given by (\ref{typicalsolutionbigrid}) into (\ref{p2EulerLagrange}), we see that the control
$\tilde{v}_h(t)$ satisfies the identity
\begin{align}\label{p2EulerLagrangeFourierBigrid}\int\limits_0^T\tilde{v}_h(t)\exp(is\lambda_h^{a,k}t)\,dt&=\frac{(-1)^k}{n^{a,k}}\frac{\lambda_h^{a,k}}{\frac{\sin(k\pi h)}{h}}
\left(\frac{\widehat{y}^{a,k,1}_h}{\lambda_h^{a,k}}-is\widehat{y}^{a,k,0}_h\right)\\\nonumber&+
\frac{is(-1)^k}{n^{a,k}}\frac{\lambda_h^{a,k}}{\frac{\sin(k\pi h)}{h}}(E_{h,1}^k+E_{h,2}^k)\exp(is\lambda_h^{a,k}T),\end{align}
where the error terms are
$$E_{h,1}^k=-(-1)^k\frac{\sin(k\pi h)}{h}\int\limits_0^T\tilde{v}_h(t)(\widehat{u}^{o,k}_{lo}(t)n^{o,k}+
\widehat{u}^{a,k}_{hi}(t)n^{a,N+1-k}+\widehat{u}^{o,k}_{hi}(t)n^{o,N+1-k})\,dt$$
and
$$E_{h,2}^k=\widehat{y}^{o,k}_{h,1}\widehat{u}^{o,k}_{lo}(0)+
\widehat{y}^{a,N+1-k}_{h,1}\widehat{u}^{a,k}_{hi}(0)+\widehat{y}^{o,N+1-k}_{h,1}\widehat{u}^{o,k}_{hi}(0)
-\widehat{y}^{o,k}_{h,0}\widehat{u}^{o,k}_{lo,t}(0)-\widehat{y}^{a,N+1-k}_{h,0}\widehat{u}^{a,k}_{hi,t}(0)-
\widehat{y}^{o,N+1-k}_{h,0}\widehat{u}^{o,k}_{hi,t}(0).$$

Passing to the limit as $h\to 0$ in the left hand side and in the first term in the right hand side of (\ref{p2EulerLagrangeFourierBigrid})
can be done as for the truncated space $\mathcal{S}_h$ in (\ref{p2SubspTruncation}). Therefore, in order to prove that
the weak limit $\tilde{v}^*$ satisfies (\ref{identitycontrolFourier}), it is enough to show that the error terms are small as $h\to 0$, i.e.
\begin{equation}|E_{h,1}^k|\to0\mbox{ and }|E_{h,2}^k|\to 0\mbox{ as }h\to 0.\label{ConvErrorBigrid}\end{equation}
From the fact that the $L^2(0,T)$-norm of the discrete control $\tilde{v}_h(t)$ is uniformly bounded as $h\to 0$, the Cauchy-Schwartz inequality and the explicit expressions of 
$\widehat{u}^{o,k}_{lo}(t)$, $\widehat{u}^{a,k}_{hi}(t)$ and $\widehat{u}^{o,k}_{hi}(t)$, we obtain
$$|E_{h,1}^k|\leq C\sqrt{T}\big(E_{h,1}^{k,1}+E_{h,1}^{k,2}+E_{h,1}^{k,3}\big)^{1/2},$$
with
$$E_{h,1}^{k,1}=|n^{o,k}|^2\Big|\frac{W_4(\Lambda^{a,k})}{W_4(\Lambda^{o,k})}\Big|^2\Big(\frac{\sin^2(k\pi h)}{\Lambda^{a,k}}+
\frac{\sin^2(k\pi h)}{\Lambda^{o,k}}\Big),$$
$$E_{h,1}^{k,2}=|n^{a,N+1-k}|^2\Big|\frac{W_5(\Lambda^{a,N+1-k})}{W_5(\Lambda^{a,k})}
\Big|^2\Big(\frac{\sin^2(k\pi h)}{\Lambda^{a,k}}+
\frac{\sin^2(k\pi h)}{\Lambda^{a,N+1-k}}\Big)$$
and $$E_{h,1}^{k,3}=|n^{o,N+1-k}|^2\Big|\frac{W_4(\Lambda^{a,N+1-k})}{W_4(\Lambda^{o,N+1-k})}\Big|^2\Big|\frac{W_5(\Lambda^{a,N+1-k})}{W_5(\Lambda^{a,k})}
\Big|^2\Big(\frac{\sin^2(k\pi h)}{\Lambda^{a,k}}+
\frac{\sin^2(k\pi h)}{\Lambda^{o,N+1-k}}\Big),$$
where $W_4$ and $W_5$ are as in (\ref{BasisBigrid}). 

On the other hand, since the $||\cdot||_{\mathcal{V}_h'}$ - norm of the initial data
$(\mathbf{Y}_h^1,\mathbf{Y}_h^0)$ to be controlled is uniformly bounded as $h\to 0$, we see that
$$|E_{h,2}^k|\leq C(E_{h,2}^{k,1}+E_{h,2}^{k,2}+E_{h,2}^{k,3})^{1/2},$$
with
$$E_{h,2}^{k,1}=\Lambda_h^{o,k}|\widehat{u}^{o,k}_{lo}(0)|^2+|\widehat{u}^{o,k}_{lo,t}(0)|^2=
\Big|\frac{W_4(\Lambda^{a,k})}{W_4(\Lambda^{o,k})}\Big|^2\Big(\frac{\Lambda^{o,k}}{\Lambda^{a,k}}+1\Big),$$
$$E_{h,2}^{k,2}=\Lambda_h^{a,N+1-k}|\widehat{u}^{a,k}_{hi}(0)|^2+|\widehat{u}^{a,k}_{hi,t}(0)|^2=
\Big|\frac{W_5(\Lambda^{a,N+1-k})}{W_5(\Lambda^{a,k})}\Big|^2\Big(\frac{\Lambda^{a,N+1-k}}{\Lambda^{a,k}}+1\Big)$$
and
$$E_{h,2}^{k,3}=\Lambda_h^{o,N+1-k}|\widehat{u}^{o,k}_{hi}(0)|^2+|\widehat{u}^{o,k}_{hi,t}(0)|^2=
\Big|\frac{W_4(\Lambda^{a,N+1-k})}{W_4(\Lambda^{o,N+1-k})}\Big|^2\Big|\frac{W_5(\Lambda^{a,N+1-k})}{W_5(\Lambda^{a,k})}\Big|^2\Big(\frac{\Lambda^{o,N+1-k}}{\Lambda^{a,k}}+1\Big).$$
For a fixed $k\in\nn$, let us study the convergence as $h\to 0$ of the terms $E_{h,j}^{k,i}$, $1\leq i\leq3$, $1\leq j\leq 2$:
\begin{itemize}
 \item From $|n^{o,k}|^2\to 3\tilde{W}(60)=10$ ($\tilde{W}$ introduced in (\ref{normh0})), $|W_4(\Lambda^{a,k})|^2\to |W_4(0)|^2=0$,
$|W_4(\Lambda^{o,k})|^2\to |W_4(60)|^2=24/25$, $\sin^2(k\pi h)/\Lambda^{a,k}\to 1$ and $\sin^2(k\pi h)/\Lambda^{o,k}\to 0$, we see that
 $E_{h,1}^{k,1}\to 0$ as $h\to 0$.
\item From $|n^{a,N+1-k}|^2\to 3\tilde{W}(10)=0$, $|W_5(\Lambda^{a,N+1-k})|^2\to |W_5(10)|^2=0$, $|W_5(\Lambda^{a,k})|^2\to|W_5(0)|^2=1/96$
and $\sin^2(k\pi h)/\Lambda^{a,N+1-k}\to 0$, we see that $E_{h,1}^{k,2}\to 0$ as $h\to 0$.
\item Remark that $|n^{o,N+1-k}|^2\to 3\tilde{W}(12)=6$,
$|W_4(\Lambda^{o,N+1-k})|^2\to|W_4(12)|^2=72$, $\sin^2(k\pi h)/\Lambda^{o,N+1-k}\to 0$, but $|W_4(\Lambda^{a,N+1-k})|^2\to |W_4(10)|^2=\infty$. Nevertheless,
$|W_4(\Lambda^{a,N+1-k})|^2|W_5(\Lambda^{a,N+1-k})|^2\to 0$, so that at the end $E_{h,1}^{k,3}\to 0$ as $h\to 0$.
\item Remark that $\Lambda^{o,k}\to 60$ and $\Lambda^{a,k}\to 0$. Observe that $|W_4(\Lambda^{a,k})|^2\to 0$ like
$|\Lambda^{a,k}|^2$. This compensates the singularity of $1/\Lambda^{a,k}$ so that $E_{h,2}^{k,1}\to 0$ as $h\to 0$.
\item $|W_5(\Lambda^{a,N+1-k})|^2\to 0$ since it involves the factor $(10-\Lambda^{a,N+1-k})^3\sim \sin^6(k\pi h/2)$. This
cancels the singularity introduced by $1/\Lambda^{a,k}\sim 1/\sin^2(k\pi h/2)$ and ensures that $E_{h,2}^{k,2}\to 0$ as $h\to 0$.
\item $|W_4(\Lambda^{a,N+1-k})|^2|W_5(\Lambda^{a,N+1-k})|^2\to 0$ since it contains the factor 
$(10-\Lambda^{a,N+1-k})^2\sim \sin^4(k\pi h/2)$ which compensates the singularity of $1/\Lambda^{a,k}\sim 1/\sin^2(k\pi h/2)$ so that
$E_{h,2}^{k,3}\to 0$ as $h\to 0$.
\end{itemize}
This concludes (\ref{ConvErrorBigrid}) and the fact that the weak limit $\tilde{v}^*$ of the sequence of discrete HUM controls obtained by minimizing the functional
$\mathcal{J}_h$ over the bi-grid class $\mathcal{S}_h$ in (\ref{p2SubsBigrid}) is a control
in the continuous problem (\ref{contwavecontrolled}).

\textbf{Step II - The weak limit $\tilde{v}^*$ is the normal derivative of a solution of the continuous adjoint problem (\ref{contwaveadjoint}).}
Consider $(\mathbf{\tilde{U}}_h^0,\mathbf{\tilde{U}}_h^1)\in\mathcal{S}_h$ (which in what follows can be both
the subspace in (\ref{p2SubspTruncation}) or the one in (\ref{p2SubsBigrid})) to be the minimum of the functional $\mathcal{J}_h$. Due to the uniform nature of the observability inequality
(\ref{p2ObservabilityInequalityGeneral}), $\mathcal{E}_h(\mathbf{\tilde{U}}_h^0,\mathbf{\tilde{U}}_h^1)$ is uniformly bounded, i.e. there exists
a constant $C>0$ independent of $h$ such that
\begin{equation}\mathcal{E}_h(\mathbf{\tilde{U}}_h^0,\mathbf{\tilde{U}}_h^1)
=\frac{1}{2}\sum\limits_{k=1}^{N}(\Lambda_h^{a,k}|\widehat{\tilde{u}}^{a,k,0}_h|^2+|\widehat{\tilde{u}}^{a,k,1}_h|^2
+\Lambda_h^{o,k}|\widehat{\tilde{u}}^{o,k,0}_h|^2+|\widehat{\tilde{u}}^{o,k,1}_h|^2)\leq C.\label{UnifBoundMinimizer}\end{equation}
Due to property (\ref{nullresonant}), the resonant mode in the solution of the adjoint problem (\ref{p2adjoint}) for initial data in the filtered space $\mathcal{S}_h$ in (\ref{p2SubspTruncation}) or  (\ref{p2SubsBigrid}) vanishes, so that the Fourier representation of the total energy in the left hand side of (\ref{UnifBoundMinimizer}) is valid for both filtered spaces $\mathcal{S}_h$ in (\ref{p2SubspTruncation}) and (\ref{p2SubsBigrid}). 

Remark however that the high frequency components vanish
for data in the truncation subspace $\mathcal{S}_h$ in (\ref{p2SubspTruncation}). On the other hand, for data in the bi-grid space $\mathcal{S}_h$ in (\ref{p2SubsBigrid}),  the relations between the optic and the acoustic modes and the high frequencies
in the acoustic mode and the lowest ones described in Propositions \ref{PropLinearData} and \ref{PropLowHighFreqAcoustical} hold. 

From (\ref{UnifBoundMinimizer}),
we deduce that
\begin{equation}(\lambda_h^{a,k}\widehat{\tilde{u}}^{a,k,0}_h)_k\rightharpoonup (\lambda^k\widehat{\tilde{u}}^{*,k,0})_k,\  
(\widehat{\tilde{u}}^{a,k,1}_h)_k\rightharpoonup (\widehat{\tilde{u}}^{*,k,1})_k, \ 
\Big(\frac{\widehat{\tilde{u}}^{o,k,1}_h}{\lambda_h^{o,k}}\Big)_k,
\ (\widehat{\tilde{u}}^{o,k,0}_h)_k\rightharpoonup 0\mbox{ as }h\to 0,\mbox{ weakly in }\ell^2.
\label{WeakConvOptimalData}\end{equation}

Set $\tilde{u}^{*,i}(x):=\sum_{k=1}^{\infty}\widehat{\tilde{u}}^{*,k,i}\varphi^k(x)$, observe that $(\tilde{u}^{*,0},\tilde{u}^{*,1})\in\mathcal{V}$ and
denote by $\tilde{u}^*(x,t)$ the corresponding solution of (\ref{contwaveadjoint}). Firstly, we prove that
\begin{equation}-\frac{\tilde{U}_N(t)}{h}\rightharpoonup \tilde{u}^*_x(1,t)\mbox{ as }h\to0\mbox{ weakly in }L^2(0,T).\label{p2WeakConvDiscNormalDeriv}\end{equation}

In fact, for arbitrary functions $\psi\in L^2(0,T)$ and $\psi_{\epsilon}\in C^k_0(0,T)$, we will prove the following estimate:
\begin{align}\label{EstimateWeakConv}\Big|\int\limits_0^T\Big(-\frac{\tilde{U}_N(t)}{h}-\tilde{u}^*_x(1,t)\Big)\psi(t)\,dt\Big|&
\leq \Big|\int\limits_0^T\Big(-\frac{\Gamma^a\tilde{U}_N(t)}{h}-\tilde{u}^*_x(1,t)\Big)\psi(t)\,dt\Big|\\&+C||\psi-\psi_{\epsilon}||_{L^2(0,T)}+
Ch^k||\psi^{(k)}_{\epsilon}||_{L^2(0,T)},\nonumber\end{align}
where $\Gamma^a:=\Gamma^a_1$ is the projection on the acoustic branch defined by (\ref{p2projacoudelta}). In a similar way, we define the projection
on the optic branch, $\Gamma^{o}$.  In order to prove (\ref{EstimateWeakConv}), we decompose its right hand side as follows:
\begin{align}\nonumber\int\limits_0^T\Big(-\frac{\tilde{U}_N(t)}{h}-\tilde{u}^*_x(1,t)\Big)\psi(t)\,dt&=
\int\limits_0^T\Big(-\frac{\Gamma^{a}\tilde{U}_N(t)}{h}-\tilde{u}^*_x(1,t)\Big)\psi(t)\,dt
+\int\limits_0^T\Big(-\frac{\Gamma^{o}\tilde{U}_N(t)}{h}\Big)(\psi(t)-\psi_{\epsilon}(t))\,dt
\\\nonumber&+\int\limits_0^T\Big(-\frac{\Gamma^{o}\tilde{U}_N(t)}{h}\Big)\psi_{\epsilon}(t)\,dt=I_h^1+I_h^2+I_h^3.\end{align}
Taking into account that $\psi_{\epsilon}\in C_c^k(0,T)$, by integration by parts, we have
$$I_h^3=(-1)^k\int\limits_0^T\psi_{\epsilon}^{(k)}(t)
\Big(\sum\limits_{\pm}\sum\limits_{l=1}^N\widehat{\tilde{u}}^{o,l}_{\pm}\frac{1}{(\pm i\lambda_h^{o,l})^k}
\exp(\pm i\lambda_h^{o,l}(t-T))\Big(-\frac{\varphi^{o,l}_N}{h}\Big)\Big)\,dt.$$
From the Cauchy-Schwartz and the admissibility inequality, the bound (\ref{UnifBoundMinimizer}) and the fact that $\Lambda_h^{o,l}\geq 12/h^2$ for all $1\leq l\leq N$, we deduce that
$$|I_h^2|\leq \Big|\Big|-\frac{\Gamma^{o}\tilde{U}_N}{h}\Big|\Big|_{L^2(0,T)}||\psi-\psi_{\epsilon}||_{L^2(0,T)}
\lesssim||\psi-\psi_{\epsilon}||_{L^2(0,T)}\Big(\sum\limits_{l=1}^N\Lambda_h^{o,l}(|\widehat{\tilde{u}}^{o,l}_+|^2+
|\widehat{\tilde{u}}^{o,l}_-|^2)\Big)^{1/2}\lesssim ||\psi-\psi_{\epsilon}||_{L^2(0,T)}$$
and
\begin{align}\nonumber |I_h^3|&\leq ||\psi_{\epsilon}^{(k)}||_{L^2(0,T)}\Big|\Big|\sum\limits_{\pm}\sum\limits_{l=1}^N\widehat{\tilde{u}}^{o,l}_{\pm}\frac{1}{(\pm i\lambda_h^{o,l})^k}
\exp(\pm i\lambda_h^{o,l}(\cdot-T))\Big(-\frac{\varphi^{o,l}_N}{h}\Big)\Big|\Big|_{L^2(0,T)}\\\nonumber&
\lesssim ||\psi_{\epsilon}^{(k)}||_{L^2(0,T)}\Big(\sum\limits_{l=1}^N\Lambda_h^{o,l}(|\widehat{\tilde{u}}^{o,l}_+|^2+
|\widehat{\tilde{u}}^{o,l}_-|^2)(\Lambda_h^{o,l})^{-k}\Big)^{1/2}\lesssim h^k||\psi_{\epsilon}^{(k)}||_{L^2(0,T)}.\end{align}

Once we get (\ref{EstimateWeakConv}), we conclude (\ref{p2WeakConvDiscNormalDeriv}) by using the following three ingredients:
\begin{itemize}
\item the weak convergence (\ref{p2ConvDataContrAcoustical}) combined with the strong convergence $n^{a,l}\to\sqrt{2}$, $\sin(l\pi h)/\lambda^{a,l}\to 1$ and
$$\int\limits_0^T\psi(t)\exp(\pm i(t-T)\lambda_h^{a,l})\,dt\to \int\limits_0^T\psi(t)\exp(\pm i(t-T)\lambda^l)\,dt\mbox{ as }h\to 0,$$
allowing us to pass to the limit as $h\to 0$ in the sense of $\ell^2$ in the sum in the right hand side of
$$\int\limits_0^T\Big(-\frac{\Gamma^{a}\tilde{U}_N(t)}{h}\Big)\psi(t)\,dt=\sum\limits_{\pm}\sum\limits_{l=1}^N
\lambda_h^{a,l}\widehat{\tilde{u}}^{a,l}_{\pm}n^{a,l}(-1)^l\frac{\sin(l\pi h)}{\lambda^{a,l}}\int\limits_0^T\psi(t)\exp(\pm i(t-T)\lambda_h^{a,l})\,dt,$$
so that we can guarantee that the first term in the right hand side of (\ref{EstimateWeakConv}) is small as $h\to 0$.
\item the density of $C_c^k(0,T)$ in $L^2(0,T)$, allowing to choose $\epsilon$ so that $||\psi-\psi_{\epsilon}||_{L^2(0,T)}$ is arbitrarily small.
\item an appropriate choice of the mesh size $h$ according to $\epsilon$, so that $h^k||\psi^{(k)}_{\epsilon}||_{L^2(0,T)}$ is arbitrarily small.
\end{itemize}

Let us check that $\tilde{v}^*=\tilde{u}^*_x(1,t)$ in $L^2(0,T)$. Indeed, using (\ref{p2WeakConvControl}) and
then (\ref{p2WeakConvDiscNormalDeriv}), we have:
$$\int\limits_0^T\tilde{v}^*(t)\psi(t)\,dt=\lim\limits_{h\to 0}\int\limits_0^T\tilde{v}_h(t)\psi(t)\,dt=\lim\limits_{h\to 0}
\int\limits_0^T\Big(-\frac{\tilde{U}_N(t)}{h}\Big)\psi(t)\,dt=\int\limits_0^T\tilde{u}^*_x(1,t)\psi(t)\,dt,$$
for an arbitrary $\psi\in L^2(0,T)$. Therefore, $\tilde{v}^*=\tilde{v}$ in $L^2(0,T)$ and also $\tilde{u}^*_x(1,t)=\tilde{u}_x(1,t)$, which, jointly with the continuous
observability inequality (\ref{contobsineq}) gives that $(\tilde{u}^{*,0},\tilde{u}^{*,1})=(\tilde{u}^0,\tilde{u}^1)$ in $\mathcal{V}$ and then
$(\widehat{\tilde{u}}^{*,k,i})_k=(\widehat{\tilde{u}}^{k,i})_k$ in $\ell^2$, $i=0,1$. This means that once we have identified that the weak limit of the
discrete controls is the continuous HUM control, we have the $\Gamma$-convergence of the discrete minimizer to the continuous one.

\textbf{Step III - Strong convergence of the discrete controls. } In order to prove that $\tilde{v}_h$ converges strongly in $L^2(0,T)$ to $\tilde{v}$ as
$h\to 0$, it is enough to prove that
\begin{equation}\lim\limits_{h\to 0}\int\limits_0^T|\tilde{v}_h(t)|^2\,dt=\int\limits_0^T|\tilde{v}(t)|^2\,dt.\label{p2ConvL2normContr}\end{equation}
Using as test function $\mathbf{U}_h(t)$ in (\ref{p2EulerLagrange}) the minimizer $\mathbf{\tilde{U}}_h(t)$, we have that
\begin{equation}\int\limits_0^T|\tilde{v}_h(t)|^2\,dt=\int\limits_{0}^T\Big|\frac{\tilde{U}_N(t)}{h}\Big|^2\,dt=
\langle(\mathbf{Y}_h^1,-\mathbf{Y}_h^0),
(\mathbf{\tilde{U}}_h(0),\mathbf{\tilde{U}}_{h,t}(0))\rangle_{\mathcal{V}_h',\mathcal{V}_h}=P_h^{a}+P_h^{o},\label{p2EulerLagrangePartic}\end{equation}
where
$$P_h^{\alpha}=\sum\limits_{\pm}\sum\limits_{k=1}^N\frac{1}{2}
\Big(\frac{\widehat{y}^{\alpha,k,1}_h}{i\lambda_h^{\alpha,k}}\mp\widehat{y}^{\alpha,k,0}_h\Big)
\big(i\lambda_h^{\alpha,k}\widehat{\tilde{u}}^{\alpha,k,0}_h\pm \widehat{\tilde{u}}^{\alpha,k,1}_h\big)\exp(\mp iT\lambda_h^{\alpha,k}).$$

Let us remark that using the strong convergence of the acoustic part of the initial data to be controlled (\ref{p2ConvDataContrAcoustical}) and the
boundedness of the energy of the minimizer of $\mathcal{J}_h$,
$$\lim\limits_{h\to 0}P_h^{a}=\sum\limits_{\pm}\sum\limits_{k=1}^{\infty}\frac{1}{2}
\Big(\frac{\widehat{y}^{k,1}}{i\lambda^k}\mp\widehat{y}^{k,0}\Big)
\big(i\lambda^k\widehat{\tilde{u}}^{k,0}\pm \widehat{\tilde{u}}^{k,1}\big)\exp(\mp iT\lambda^k)=\int\limits_0^T|\tilde{u}_x(1,t)|^2\,dt=
\int\limits_0^T|\tilde{v}(t)|^2\,dt.$$
On the other hand, taking into account the uniform boundedness of the energy of the minimizer of $\mathcal{J}_h$, we get
$$|P_h^o|^2\leq C\sum\limits_{k=1}^N\Big(\frac{|\widehat{y}^{o,k,1}_h|^2}{\Lambda_h^{o,k}}+|\widehat{y}^{o,k,0}_h|^2\Big)\to 0\mbox{ as }h\to 0,$$
which concludes (\ref{p2ConvL2normContr}) and the strong convergence of the optimal control.
\end{proof}

\begin{remark}In \cite{ErvZuaSurv}, it was proved that, for initial data $(y^0,y^1)$ in the continuous control problem (\ref{contwavecontrolled}) belonging to the more regular space $H_0^1\times L^2(0,1)$, the numerical controls obtained for the finite difference or the linear finite element semi-discretization of the wave equation (\ref{contwavecontrolled}) converge to the continuous HUM controls with an error order $h^{2/3}$. This is due to the fact that $|\lambda_h(\xi)-\xi|\sim h^2\xi^3$, where $\lambda_h(\xi)$ can be each one of the dispersion relations for the finite difference or finite element approximation of the wave equation. In \cite{MarZuaP2proc}, we observed the fact that the acoustic dispersion relation $\lambda_h^a(\xi)$ of the $P_2$ - finite element method approximates the continuous one $\xi$ with error order $h^4\xi^5$ for all $\xi\in[0,\pi/h]$, so that the convergence error for the numerical controls obtained by the bi-grid algorithm in the quadratic approximation of the wave equation increases to $h^{4/5}$ under the same regularity assumptions on the continuous initial data to be controlled.  \label{remark5}\end{remark}

\section{Final comments and open problems}\label{SectOpenPbms}
\begin{itemize}\item All the results in this paper can be extended to finite element methods of arbitrary order $k$, with the additional difficulty that when
computing the eigenvalues, the quadratic equation (\ref{p2spectral6}) has to be replaced by a $k$-th order algebraic equation in $\Lambda$ which
is technically complicated to be solved explicitly. The same difficulty arises when passing to several space dimensions.

\item The results in \cite{ErvZheZua} providing a general method to obtain uniform observability results for time discretizations of conservative
system lead to the extension of our observability results for the $P_2$ - space semi-discretization to fully discrete $P_2$ conservative approximations of the wave equation.

\item The extension of the results in this paper to non-uniform meshes is a completely open problem.
\item The last open problem we propose is related to \cite{IgZuaSch} and concerns the dispersive properties of the Schr\"{o}dinger equation approximated in space using the $P_2$ - finite element method. Designing appropriate bi-grid algorithms taking care 
of all the singularities of both group velocity and acceleration simultaneously  is an open problem.
\end{itemize}

\end{document}